\documentclass{amsart}
\usepackage[misc]{ifsym}
\usepackage{hyperref}
\hypersetup{hidelinks,colorlinks,linkcolor=blue,urlcolor=black}

\newtheorem{theorem}{Theorem}[section]

\theoremstyle{definition}

\theoremstyle{remark}

\numberwithin{equation}{section}

\usepackage{graphicx}
\usepackage{epstopdf}
\usepackage{algorithm} 
\usepackage{algpseudocode} 
\usepackage{amsfonts,amsmath,amssymb,mathtools,amsthm} 
\usepackage[mathscr]{eucal}
\usepackage{graphicx,color,xcolor}
\usepackage{enumitem}
\usepackage{cite}
\usepackage{nicefrac}
\usepackage{tikz}
\usepackage{array}

\usepackage[noabbrev, nameinlink]{cleveref}

\newtheorem{asmpt}[theorem]{Assumption}
\crefname{asmpt}{\textup{Assumption}}{\textup{Assumptions}}
\creflabelformat{asmpt}{#2\textup{#1}#3} 

\crefname{theorem}{\textup{Theorem}}{\textup{Theorems}}
\creflabelformat{theorem}{#2\textup{#1}#3} 

\crefname{lemma}{\textup{Lemma}}{\textup{Lemmas}}
\creflabelformat{lemma}{#2\textup{#1}#3} 
\crefname{proposition}{\textup{Proposition}}{\textup{Propositions}}
\creflabelformat{proposition}{#2\textup{#1}#3} 
\crefname{cor}{\textup{Corollary}}{\textup{Corollaries}}
\creflabelformat{cor}{#2\textup{#1}#3} 
\crefname{definition}{\textup{Definition}}{\textup{Definitions}}
\creflabelformat{definition}{#2\textup{#1}#3} 
\crefname{remark}{\textup{Remark}}{\textup{Remarks}}
\creflabelformat{remark}{#2\textup{#1}#3} 
\crefname{figure}{\textup{Figure}}{\textup{Figures}}
\creflabelformat{figure}{#2\textup{#1}#3}
\crefname{table}{\textup{Table}}{\textup{Tables}}
\creflabelformat{table}{#2\textup{#1}#3} 
\crefname{algorithm}{\textup{Algorithm}}{\textup{Algorithms}}
\creflabelformat{algorithm}{#2\textup{#1}#3} 
\crefname{section}{\textup{Section}}{\textup{Sections}}
\creflabelformat{section}{#2\textup{#1}#3} 

\usepackage{subcaption}

\numberwithin{equation}{section}
\numberwithin{algorithm}{section}
\numberwithin{figure}{section}

\makeatletter
\let\@citexOld\@citex
\def\@citex[#1]#2{\textup{\@citexOld[#1]{#2}}}
\makeatother

\usepackage[draft,todonotes={textsize=scriptsize}]{changes}
\definechangesauthor[color=red]{JB}

\usepackage[%
left=1.25in,
right=1.5in,
bottom=1in,
top=1in,
footskip=0.5in,
]{geometry}

\newcommand{\bB}{{\boldsymbol B}}

\newcommand{\bU}{{\boldsymbol U}}

\newcommand{\bY}{{\boldsymbol Y}}

\newcommand{\ba}{{\boldsymbol a}}
\newcommand{\bb}{{\boldsymbol b}}
\newcommand{\bc}{{\boldsymbol c}}

\newcommand{\bh}{{\boldsymbol h}}
\newcommand{\bp}{{\boldsymbol p}}
\newcommand{\bq}{{\boldsymbol q}}
\newcommand{\br}{{\boldsymbol r}}

\newcommand{\bv}{{\boldsymbol v}}
\newcommand{\bx}{{\boldsymbol x}}
\newcommand{\bz}{{\boldsymbol z}}

\newcommand{\Js}{{\mathscr J}}

\newcommand{\Dc}{{\mathcal D}}
\newcommand{\Lc}{{\mathcal L}}
\newcommand{\Nc}{{\mathcal N}}
\newcommand{\Sc}{{\mathcal S}}

\newcommand{\Eb}{{\mathbb E}}

\newcommand{\Nb}{{\mathbb N}}

\newcommand{\Rb}{{\mathbb R}}

\newcommand{\xmax}{{x_{\max}}}
\newcommand{\vmax}{{v_{\max}}}

\newcommand{\vd}{{\bar v_{\mathrm d}}}
\newcommand{\xd}{{\bar x_{\mathrm d}}}
\newcommand{\zd}{{\bar \bz_{\mathrm d}}}

\newcommand{\zkappajk}{\bz^\kappa_{j}}
\newcommand{\Deltatkappa}{\Delta t^\kappa}

\newcommand{\blambda}{{\boldsymbol \lambda}}
\newcommand{\bmu}{{\boldsymbol \mu}}
\newcommand{\bphi}{{\boldsymbol \phi}}

\newcommand{\rmd}{{\, \mathrm{d}}}

\newcommand{\tol}{{\mathrm{tol}}}

\allowdisplaybreaks
\makeatletter
\renewenvironment{algorithm}
{%
	\bigskip
		\begin{center}
			\refstepcounter{algorithm}
			\hrule height.8pt depth0pt \kern2pt
			\renewcommand{\caption}[2][\relax]{
				{\raggedright\textbf{\fname@algorithm~\thealgorithm} ##2\par}%
				\ifx\relax##1\relax 
				\addcontentsline{loa}{algorithm}{\protect\numberline{\thealgorithm}##2}%
				\else 
				\addcontentsline{loa}{algorithm}{\protect\numberline{\thealgorithm}##1}%
				\fi
				\kern2pt\hrule\kern2pt
			}
	}
	{%
		\kern2pt\hrule\relax
		\bigskip
	\end{center}
}
\makeatother

\begin{document}

\title{Adjoint-based optimal control of jump-diffusion processes}

\author[Jan Bartsch, Alfio Borz\`i, Gabriele Ciaramella, and Jan Reichle]{}

%
%
%

\thanks{The work was partially funded by Deutsche Forschungsgemeinschaft (DFG) within SFB 1432, Project-ID 425217212, and the Young Scholar Fund by of the University of Konstanz}

\thanks{$^*$Corresponding author: Jan Bartsch}

\subjclass[2020]{93E03, 
	93D15, 
	65C05 
 }

\keywords{Optimization of SDEs, Jump-diffusion process, First-order optimality conditions, Monte Carlo methods, Stochastic gradient methods}

\date{\today \, (submitted)}


\begin{abstract}
	Stochastic differential equations (SDEs) using jump-diffusion processes describe many natural phenomena at the microscopic level.
	Since they are commonly used to model economic and financial evolutions, the calibration and optimal control of such processes are of interest to many communities and have been the subject of extensive research.
	In this work, we develop an optimization method working at the microscopic level.
	This allows us also to reduce computational time since we can parallelize the calculations and do not encounter the so-called curse of dimensionality
	that occurs when lifting the problem to its macroscopic counterpart using partial differential
	equations (PDEs). 
	Using a discretize-then-optimize approach, we derive an adjoint process
	and an optimality system in the Lagrange framework. 
	Then, we apply Monte Carlo methods to solve all the arising equations. 
	We validate our optimization strategy by extensive numerical experiments.
	We also successfully test a optimization procedure that avoids storing the information of the forward equation.
\end{abstract}

\maketitle

\centerline{\scshape
	Jan Bartsch$^{{\href{mailto:jan.bartsch@uni-wuerzburg.de}{\textrm{\Letter}}}*1}$,
	Alfio Borz\`i$^{{\href{mailto:alfio.borzi@uni-wuerzburg.de}{\textrm{\Letter}}}2}$,
	Gabriele Ciaramella $^{{\href{mailto:gabriele.ciaramella@polimi.it}{\textrm{\Letter}}}3}$,
	and Jan Reichle$^{{\href{mailto:jan.reichle@uni-konstanz.de}{\textrm{\Letter}}}4}$}

\medskip

{\footnotesize
	\centerline{$^1$Institut f\"ur Mathematik, Universit\"at W\"urzburg,}
	\centerline{ Emil-Fischer-Str. 40, 97074 W\"urzburg, Germany}
}

\medskip

{\footnotesize
	\centerline{$^2$Institut f\"ur Mathematik, Universit\"at W\"urzburg,}
	\centerline{ Emil-Fischer-Str. 30, 97074 W\"urzburg, Germany}
}

\medskip

{\footnotesize
	\centerline{$^3$MOX Lab, Dipartimento di Matematica, Politecnico di Milano,}
	\centerline{Piazza Leonardo da Vinci 32, 20133 Milano, Italy}
}

\medskip

{\footnotesize
	\centerline{$^3$Fachbereich f\"ur Mathematik und Statistik, Universit\"at Konstanz,}
	\centerline{Universitätsstraße 10, 78464 Konstanz, Germany}
}

\section{Introduction}

Since many physical phenomena are inherently subject to uncertainties and noise their accurate description is often achieved using stochastic differential equations (SDEs).
A very broad class of SDEs is the one governed by the so-called jump-diffusion processes. 
Their applications range from physics \cite{Bao2019levyFilter_SDEJD, DelCastilloCarreras2004fractional, Jain2021microStochPrandtlTomlison} over economics and finance \cite{BrutiPlaten2007approximation, FramstadOksendal2004SufficientMaximumPrinciple, SachsSchu213GradienComputation_ModelCalibration} to medical imaging \cite{Grenander1994KnowledgeComplexSystems}.
The jump-diffusion processes consist of three parts: 1) a deterministic drift, 2) a stochastic diffusion that is usually modeled as a Brownian motion, and 3) a stochastic jump usually realized using a compound Poisson process.
Suppose that one wants to model and control the state $\bz(t) \in \Rb^d$ of a physical system over a time interval $[0,T]$, $T>0$, then the corresponding SDE is given by
\begin{align}
	\rmd \bz(t) = \ba(\bz(t),u(\bz,t),t)\,\rmd t \, + \,
	\bb(\bz(t),t)\,\rmd \bB(t) \,
	+\,
	\bc(\bz(t^-),t)\, \rmd \bY(t)
	\qquad\qquad
	t \in (0,T].
	\label{eq:SDE_general}
\end{align}
In \eqref{eq:SDE_general}, we denote by $\bB$ a standard $d$-dimensional Brownian motion \cite[Definition 1.1.13]{PlatenBruti2010NumSolSDEFinance} and by $\bY$ a $d$-dimensional compound Poisson process \cite[p. 10]{PlatenBruti2010NumSolSDEFinance}.
The function $u(\bz,t)$ is called the control.
The equation \eqref{eq:SDE_general} is completed with an initial condition $\mathring{\bz}$ such that $\bz(0)=\mathring{\bz}$.

The optimal control of jump-diffusion processes is of interest to many communities and has therefore led to extensive research over many years; see, e.g.,  \cite{Bismut1978IntroOCPStoch, FramstadOksendal2004SufficientMaximumPrinciple, GaviraghiBorzi2016OCP_JD, KushnerDiMasi1978OCP_JD, Kushner2001NumMeth_StOCP} and references therein.
An optimal control problem is given by minimizing a certain objective $J(\bz,u)$ subject to the pair $(\bz,u)$ satisfying \eqref{eq:SDE_general} and there are in principle two approaches to solving optimal control problems governed by \eqref{eq:SDE_general}.
On the one hand, one can lift the problem to the level of partial differential equations (PDEs). 
To this end, one defines a probability density function $f(z,t)$ that contains information on the probability of the state $\bz$ being in a certain configuration $z$ at a certain timestep $t$. 
The PDE that governs the evolution of $f$ is given by the famous Fokker-Planck equation \cite{GaviraghiBorzi2016OCP_JD} that contains in the case of jump-diffusion processes also an integral part. 
In this setting, one can apply tools from PDE-constrained optimization \cite{Treoltzsch2010OCP_PDE}.
In the case of $\bb \equiv 0$ and $\bc \equiv 0$, the first author considered optimal control problems for the corresponding partial differential equation, i.e. the Liouville equation.
In \cite{Bartsch2019Theoretical,Bartsch2021Numerical} deterministic numerical methods were applied to solve the arising equations.
In \cite{Bartsch2020OCPKS,Bartsch2021MOCOKI,BartschKnopfScheurerWeber2023controllingVPMC}, the author considered a Monte-Carlo framework to solve the arising equations.

On the other hand, it is possible to work at the microscopic level and directly characterize the control using stochastic differential equations. 
Here, no probability density functions come into play and hence no PDEs must be solved.

In the PDE-based approach, the following difficulties occur. 
First, the complexity of the problem grows exponentially with the dimension of the state $\bz \in \Rb^d$. This is the so-called \emph{curse-of-dimensionality} from which all PDE-based approaches suffer.
Additionally, for jump-diffusion processes the Fokker-Planck equation becomes an integro-differential equation, requiring more advanced solution techniques than in standard PDE optimization.

The problems of the PDE approach can be circumvented to some extent when remaining on the microscopic level.
First, an increase in the state dimensions does not lead to an exponential growth of computational complexity. 
If one additionally supposes that each state component is independent of the others then it is possible to parallelize the computations with respect to the dimensions.
Furthermore, the additional presence of the Poisson process does not lead to a significantly different structure of the solution method.
In fact, the jump part can be approximated by a piecewise deterministic process for which the jump times have to be sampled according to a Poisson distribution.

The new contribution in this work is that we construct an optimization strategy working directly on the microscopic level. 
In particular, we do not have to assemble probability density functions that turn out to be a time-consuming bottleneck in hybrid (macroscopic-microscopic) methods \cite{Bartsch2021MOCOKI, Bartsch2020OCPKS}.
Furthermore, in contrast to \cite{BartschDenkVolkwein2023adjointCalibrationSDE}, we do not consider only Brownian motion, but also jump processes determined by a Poisson process.
The structure of our control is chosen in such a way that it realizes a feedback-like control.
More specifically, we consider the control $u$ as consisting of a fixed dependence on the state variable given by the smooth function $\boldsymbol \Phi(\bz)$ and a square integrable time-dependent function $\bmu(t)$. 
This time-dependent function is considered considered to be our control mechanism that we want to optimize.
Hence, the control can be written given as
\begin{align*}
	u(\bz,t) = \bmu(t)\boldsymbol{\Phi}(\bz).
\end{align*}

Furthermore, we use a tracking-type cost functional with a Gaussian-like exponential function.
This is motivated by the fact, that we want to have bounded functions with bounded derivatives in the optimality system of our framework later on.
Furthermore, this setting is then close to previous works \cite{Bartsch2021MOCOKI,BartschKnopfScheurerWeber2023controllingVPMC}.

This work is organized as follows. 
In \cref{sec:OCP}, we formulate and analyze the optimal control problem, introducing the system dynamics and the corresponding cost functional. 
Assumptions necessary for well-posedness are also discussed.
In \cref{sec:FullyDiscreteProblem}, we present the fully discretized problem. 
Furthermore, we analyze the convergence of discrete approximations to the continuous problem.
In \cref{sec:OptSys}, we derive the optimality system using an adjoint-based approach, including the equations necessary for characterizing optimal solutions.
In \cref{sec:NumImpl}, we explain the numerical implementation, including the discretization scheme, the design of shape functions, the handling of jump processes, and the optimization procedure.
In \cref{sec:NumExp}, we present numerical experiments to validate the proposed method. 
Specific test cases include centering particles, stabilization, and following a time-dependent trajectory with systems of coupled and uncoupled particles.
Moreover, we successfully test here an idea that avoids the storing of the forward trajectories.
A section of conclusion finishes this work.

\section{Formulation of the optimal control problem}
\label{sec:OCP}

We consider a finite time horizon $[0,T]$ with $T>0$ and the a system of $N \in \Nb$ particles in the phase space $\Rb^{d_z}$ consisting of velocity and position space. 
Hence, we have the whole state space $\Rb^d$ with $d\coloneqq N\,d_z$. In the following, we consider a two-dimensional phase space, i.e. $d_z=2$.
We appoint each particle with position and velocity, i.e. $\bz_j = (\bx_j,\bv_j)$ for $j \in [N]$, where we use the notation $[N]\coloneqq \{1,\ldots,N\}$.

Our goal is to design a control field capable of driving the mean of an initial configuration $\mathring{\bz}$ of particles to a desired configuration or follow a desired trajectory $\zd(t)$ on average.
For this purpose, we define the structure of the deterministic coefficient as:
\begin{align}
	\ba(\bz(t),u(\bz,t),t) = \mathrm{e} \otimes \binom{0}{u(\bz,t)} + H\bz, 
	\qquad\qquad
	H \in \Rb^{N \, d_z \times N \, d_z}
	\label{eq:deterministic_structure},
\end{align}
where we define $\mathrm{e}=(1,\ldots,1) \in \Rb^{N\,d_z}$ and denote by $\otimes$ the standard Kronecker product, i.e.
\begin{align}
	\mathrm{e} \otimes \binom{0}{u(\bz,t)}
	=
	\Big(0,u(\bz_1,t),0,u(\bz_2,t),\ldots,0,u(\bz_N,t)\Big)^\top \in \Rb^{N \, d_z}.
	\label{eq:definition_otimesJ}
\end{align}
The symbol $^\top$ denotes the transpose.
By the shape of the coefficient of the deterministic part given in \eqref{eq:deterministic_structure}, the control can be interpreted as force acting in the velocity component. 
The matrix $H$ accounts for the dynamic of the particles in the uncontrolled case. 
For example, it can model Hook's law.

To measure the effectiveness of our control mechanism, we consider the cost functional
\begin{align}
	J(\bz,u) = \Eb\left[ \int_0^T \frac{1}{N} \sum_{j=1}^{N} \Js(\bz_j,t) \rmd t + \frac{\alpha}{2} \| u \|_{L^2(\Rb^d \times [0,T])}^2  \right],
	\label{eq:continous_objective}
\end{align}
where $\alpha>0$ is the control weight and states the relative importance of the cost of the control in the optimization problem.
The symbol $\Eb$ denotes the expectation value with respect to the stochastic variable.
On the functional $\Js$ in \eqref{eq:continous_objective}, we impose the following
\begin{asmpt}
	\label{asmpt:functional}
	\begin{enumerate}[label=\textup{\arabic*)}]
		\item The functional $\Js: \Rb^d \rightarrow \Rb$ is lower semicontinuous and bounded from below.
		\item The functional $\Js$ is continuously differentiable, i.e. $\Js \in C^1(\Rb^d \times [0,T])$.
	\end{enumerate}
\end{asmpt}

\bigskip

We consider the following form of the control for some $L \in \Nb$ that realizes a separation in phase space and time
\begin{align}
	u: \Rb^d \times [0,T] \rightarrow \Rb,
	\qquad\qquad
	u(\bz(t),t) = \sum_{\ell=1}^L \bmu_\ell(t) \, \bphi_\ell(\bz(t)),
	\label{eq:control_structure}
\end{align}
with $\bmu_\ell: [0,T] \rightarrow \Rb$ for $\ell \in [L]$.
In this work, the shape functions $\bphi_\ell:\Rb^d \rightarrow \Rb$ in phase space are given and $\bmu_\ell$ are the optimization variables.
We assume the following structure of $\bphi_\ell$ with given shape functions $\bphi^x_\ell$ and $\bphi^v_\ell$ in position and velocity, respectively,
\begin{align*}
	\bphi_\ell(\bz(t)) \coloneqq  \bphi^x_\ell(x(t)) \, \bphi^v_\ell(v(t)).
\end{align*}

We define $\bphi \coloneqq (\bphi_1,\ldots,\bphi_L)$ and $\bmu = (\bmu_1,\ldots,\bmu_L)$.
Then, we can write \eqref{eq:control_structure} as
\begin{align}
	u(\bz(t)) =  \bmu(t)^\top \bphi(\bz(t)).
\end{align}

We define
\begin{align}
	\Psi(\bz,\bmu) \coloneqq \mathrm{e} \otimes   \binom{0}{ \bmu^\top \bphi(\bz)}.
\end{align}
With this, we can write the coefficient $\ba$ in \eqref{eq:deterministic_structure} as 
\begin{align}
	\ba(\bz,u,t) = \Psi(\bz,\bmu)  + H \bz.
	\label{eq:compact_forward}
\end{align}

\bigskip
For the well-posedness of the state equation \eqref{eq:SDE_general}, we need the following assumptions  (cf. \cite{Sobczyk2013SDEs} and \cite{HighamKloeden2006ConvergenceStabilityJDP}).
\begin{asmpt}
	\label{asmpt:existence_uniqueness_SDE}
	The coefficients $\ba:\Rb^d \times \Rb \times [0,T] \rightarrow \Rb^d$, $\bb:\Rb^d \times [0,T]\rightarrow \Rb^d$, $\bc:\Rb^d \times [0,T] \rightarrow \Rb^d$ in \eqref{eq:SDE_general} fulfill
	\begin{enumerate}[label=\textup{\arabic*)}]
		\item {(At most quadratic growth)} There exists a constant $L>0$, such that
		\begin{align*}
			|\ba(\bz,u,t)|^2+|\bb(\bz,t)|^2 + |\bc(\bz,t)|^2 \leq L(1+|\bz|^2).
		\end{align*}
		\item {(Lipschitz continuity)} For an arbitrary $R>0$ there exists a constant $C_R>0$ such that for all $|\bz_1|\leq R$ and $|\bz_2| \leq R$
		\begin{multline*}
			|\ba(\bz_1,u,t)-\ba(\bz_2,u,t)|^2 + |\bb(\bz_1,t)-\bb(\bz_2,t)|^2
			\\
			+ |\bc(\bz_1,t)-\bc(\bz_2,t)|^2 \leq C_R|\bz_1-\bz_2|^2
		\end{multline*}
		for all $t\in [0,T]$ and $u \in \Rb$.
		%
	\end{enumerate}
\end{asmpt}

We can now formulate our optimal control problem:
\begin{subequations}
	\begin{align}
		\min_{(\bz,\bmu)} \; &\Eb\left[ \int_0^T \frac{1}{N} \sum_{j=1}^{N} \Js(\bz_j(t),t) \rmd t + \frac{\alpha}{2} \| \bmu(\cdot)^\top \bphi(\bz_j(\cdot)) \|_{L^2([0,T])}^2  \right] \eqqcolon j(\bz,\bmu)
		\\
		\text{s.t. }
		&\begin{cases}
			\rmd \bz(t) = \Big(\Psi(\bz(t),\bmu(t))  + H \bz(t) \Big)\rmd t \, + \,
			\bb(\bz(t),t)\rmd \bB(t) \, 
			\\
			\hspace{6cm} +\, \bc(\bz(t^-),t)  \rmd \bY(t)
			\quad
			t \in (0,T],
			\\
			\bz(0) = \mathring{\bz}.
		\end{cases}
		\label{eq:SDE_constraint}
	\end{align}
	\label{eq:OCP_continuous}
\end{subequations}

We state in the following theorem the existence and uniqueness of solutions to \eqref{eq:SDE_general}.
Notice that $\ba$ in the structure of \eqref{eq:compact_forward} fulfills \cref{asmpt:existence_uniqueness_SDE} for suitable $\bb$, $\bc$.
\begin{theorem}
	Let \cref{asmpt:existence_uniqueness_SDE} hold. 
	Then the equation \eqref{eq:SDE_general} has a unique solution whose almost all sample functions are continuous from the right.
	\label{thm:ExUniSDE}
\end{theorem}
For the proof, we refer to \cite[Theorem 3.14]{Sobczyk2013SDEs}. 
For further information on the functions that are continuous from the right (so-called cadlag functions) $\Dc([0,T])$, we refer the reader to, e.g.,  \cite{Billingsley1999ConvergenceProbMeasure}.
Using \cref{thm:ExUniSDE}, we can introduce the \emph{control-to-state map} $\Sc$:
\begin{align}
	\Sc: L^2(0,T)^L \rightarrow \Dc([0,T])
	\label{eq:Control-to-state-map}
\end{align}
that associates to every $\bmu \in L^2(0,T)^L$ the corresponding solution of \eqref{eq:SDE_constraint}.

Using the control-to-state map, we can introduce the \emph{reduced objective function} $\widehat{j}(\bmu)$ as
\begin{align}
	\widehat{j}(\bmu) \coloneqq j(\Sc(\bmu),\bmu).
	\label{eq:continous_reduced_functional}
\end{align}

We now discuss the existence of optimal controls in our setting.
It is in general not possible to prove existence of optimal controls for general SDE in the strong sense. On the contrary, even for SDEs without jumps one has to consider so-called \emph{relaxed controls} for which the probability space can not be fixed a-priori \cite{Buckdahn2010existenceStochasticControl, YongZhou1999StochasticControls}. In \cite{AhmedCharalambous2013StochMinJumpDiffusion}, the authors deal with the question of existence of optimal stochastic control for a quite general form of stochastic optimal controls problems with a state equation of the form \eqref{eq:SDE_general}.
However, it is possible to characterize an optimal control by certain maximum principles.
In \cref{sec:OptSys}, we use the Lagrange framework to characterize an optimal control function.

\section{Formulation and analysis of the fully discrete problem}
\label{sec:FullyDiscreteProblem}

In this work, we apply the discretize-before-optimize approach.
Hence, we show in the following how to discretize the continuous optimal control problem \eqref{eq:OCP_continuous}.

To approximate the expectation value $\Eb$ in \eqref{eq:OCP_continuous}, we choose $M\gg 1$ realizations for the Brownian motion and the Poisson process.

For each realization $m \in [M]$ and each particle $j \in [N]$, we calculate the $N_j^m$ jump times $\{T_{j,k}^m\}_{k=1}^{N_j^m}$.
Furthermore, we create a uniform grid in time with $N^u_t$ subintervals of size $\Delta \tau = \nicefrac{T}{N^u_t}$ as
\begin{align}
	[0,T] = \cup_{\kappa=0}^{N_t-1} [\tau^\kappa,\tau^{\kappa+1}],
	\qquad
	\tau^k = \kappa \, \Delta \tau.
	\label{eq:Deterministic_timestep}
\end{align}
For each particle $j$ and each realization $m$, we the consider a (possible) different time discretization generated by gridpoints which are given by the union of the $N_t^u$ deterministic splitting and the $N_j^m$ stochastic jump points.
More specifically, the gridpoints are given by
\begin{align}
	\{T_{j,1}^m,\ldots,T_{j,N_j^m}^m\} \; \cup \; \{\tau^0,\ldots,\tau^{N_t^u}\}
	\eqqcolon \{t^0,\ldots,t^{N_t^{j,m}}\}.
\end{align}

Now, we consider the discretization of the time interval into subintervals as
\begin{align}
	[0,T] = \cup_{\kappa=0}^{N^{j,m}_t-1} [t^\kappa,t^{\kappa+1}],
	\qquad\qquad
	\Deltatkappa \coloneqq	t^{\kappa+1} - t^\kappa,
\end{align}
Using the small time-intervals $\Deltatkappa$, we generate the samples of Brownian motion $\Delta \bB^\kappa_{j,m} \sim \Nc(0,\Deltatkappa)$.
In \cref{fig:Time_discretizations_intervals}, we visualize the different time discretizations in our work for a single particle and a single realization. 
In \cref{fig:Time_discretizations_trajectory}, we plot an exemplary trajectory corresponding to the time discretization given in \cref{fig:Time_discretizations_intervals}.

However, for convenience of the implementation, we want to approximate the control to be constant in equidistant intervals of a given length.
For this, we use the partition of the time interval $[0, T]$ into the $N_t^u>1$, equally-spaced subintervals given in \eqref{eq:Deterministic_timestep}.
Within one subinterval, we assume the control to be constant.
In particular, we assume that we have the continuous in time interpolation of the control given by
\begin{align}
	\bmu^h(t) \coloneqq \sum_{\kappa=0}^{N_t^u-1} \bmu^\kappa \chi_{[t^\kappa,t^{\kappa+1}]}(t),
	\qquad\qquad
	t \in [0,T].
	\label{eq:piecewise_constant_control}
\end{align}
%

%

\begin{figure}
	\begin{subfigure}[l]{0.4\textwidth}
		\includegraphics[width=\textwidth]{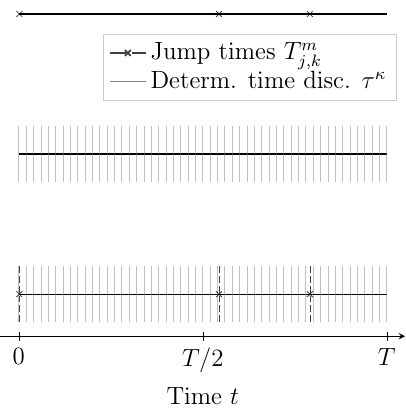}
		\caption{}
		\label{fig:Time_discretizations_intervals}
	\end{subfigure}
	\hfill
	\begin{subfigure}[r]{0.49\textwidth}
		\includegraphics[width=\textwidth]{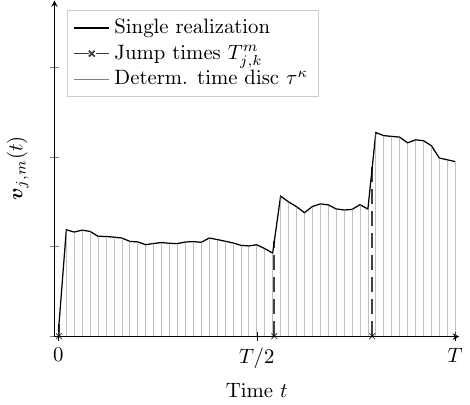}
		\caption{}
		\label{fig:Time_discretizations_trajectory}
	\end{subfigure}
	\caption{Different discretization of the time interval $[0,T]$ and exemplary trajectory. 
		Gray: Deterministic splitting in $N_t^u$ intervals; Black dashed: Stochastic splitting of the time interval into $N_j$ jump times, different for every realization and particle.
		(a) Visualization of stochastic and deterministic parts of the time discretization.
		(b) Visualization of an exemplary SDE trajectory of a single particle in a single realization together with the corresponding discretization. }
	\label{fig:Time_discretizations}
\end{figure}

The fully discrete problem is then given by
\begin{subequations}
	\begin{align}
		\min_{(\bz, \bmu)}
		&\frac{1}{M} \sum_{m=1}^M \frac{1}{N} 
		\sum_{j = 1}^N
		\sum_{k=1}^{N_j^m} \sum_{\kappa = N_t^{j,m,k}}^{N_t^{j,m,k+1}} \Js(\bz^\kappa_{j,m},t^\kappa) \, 
		+ \frac{\alpha}{2} 
		\| \bmu(t^\kappa)^\top \bphi(\bz^\kappa_{j,m}) \|_2^2 \Delta t^k
		\\
		\text{s.t. }
		&\begin{cases}
			\bz^{\kappa+1}_{j,m} = \bz^\kappa_{j,m} + \Delta t \Big( \Psi(\bz^\kappa_{j,m},\bmu(t^\kappa))  + H\bz^\kappa_{j,m}\Big)
			\\
			\hspace{4cm} + \bb(\bz^\kappa_{j,m},t^\kappa) \Delta \bB^\kappa_{j,m}
			\qquad
			\text{ for } \kappa \in [N_t^{j,m,k-1},N_t^{j,m,k}-1]
			\\
			\bz^{N_t^{j,m,k}}_{j,m} = \bc( \bz^{N_t^{j,m,k}-1}_{j,m})
			\\
			\bz^0_{j,m}(0) = \mathring{\bz}_{j,m}.
		\end{cases}
		\label{eq:model_approximation}
	\end{align}
	\label{eq:OCP_fully_discrete}
\end{subequations}

We define the continuous in time representation of a numerical solution for the $j$-th particle in the $m$-th realization as a piecewise constant function.
For this, we set $h=(M,N_t^u)$ and 
\begin{align}
	\bz_{j,m}^h(t) = \sum_{\kappa=0}^{N_t^{j,m}-1} \bz_{j,m}^\kappa \chi_{[t^\kappa,t^{\kappa+1}]}(t),
	\qquad\qquad t \in [0,T].
	\label{eq:piecewise_constant_approx}
\end{align}
With this definition, we can evaluate $\bz_{j,m}^h$ at arbitrary times, in particular at the time discretizations points of the other particles or realizations.
We set $\bz^h = \{\bz_{j,m}^h\}_{m,j=1}^{M,N}$.

We define a discretized version of the objective taking into account the time discretization of the control
\begin{align}
	j^h(\bz^h,\bmu^h) = \frac{1}{M} \sum_{m=1}^{M} \frac{\Delta t}{N} \sum_{j = 1}^N \sum_{\kappa=1}^{N_t^u}
	\Js(\bz_{j,m}^h(t^\kappa),t^\kappa) + \frac{\alpha}{2} |\bphi(\bz_{j,m}^h(t^\kappa))^\top\bmu^h(t^\kappa)|^2.
	\label{eq:discrete_functional}
\end{align}

We have that the discrete solution $\bz^h$ defined as the piecewise constant approximation (cf. \eqref{eq:piecewise_constant_approx}) converges to the true solution $\bz$ in the following sense:
\begin{theorem}
	\label{thm:ApproximationForwardSDE}
	Let \cref{asmpt:existence_uniqueness_SDE} hold.
	Then there exists for any $\bmu \in L^2(0,T)^N$ a constant $C>0$ and $N_t^* \in \Nb$ such that for all $N_t \geq N_t^*$
	\begin{align}
		\Eb \sup_{t \in [0,T]} [|\bz^h (t) - \bz(t)|^2] \leq C \Delta t(1 + \Eb[|\bz(0)|^2] ).
	\end{align}
\end{theorem}
For the proof, see 	\cite[Theorem 2.4]{HighamKloeden2006ConvergenceStabilityJDP}.
For further reading about discretization of jump-diffusion processes, we refer to \cite[Chapter 16]{HighamKloeden2021NumSDE} as well as \cite{HighamKloeden2006ConvergenceStabilityJDP} and \cite{PlatenBruti2010NumSolSDEFinance}.

Using the result of \cref{thm:ApproximationForwardSDE}, we can introduce the reduced functional
\begin{align}
	\widehat{j}^h(\bmu^h) = j(\Sc^h(\bmu^h),\bmu^h),
	\label{eq:discrete_reduced_functional}
\end{align}
where $\Sc^h(\bmu^h)$ is the solution of \eqref{eq:model_approximation} using the piecewise constant $\bmu^h \in L(0,T)^L$ given $\{\mu^\kappa\}_{\kappa=0}^{N_t^u}$ (cf. \eqref{eq:piecewise_constant_control}).
The following theorem states that convergence of the value of the discrete functional to the value of the continuous one in the case of convergence of the discrete controls to a continuous one.
Recall that we set $h=(M,N_t^u)$. 
When writing $h \rightarrow \infty$, we let both components tend to infinity at the same time.

\begin{theorem}
	Let \cref{asmpt:functional,asmpt:existence_uniqueness_SDE} hold.
	Assume that we have given a sequence $(\bmu^h)_h$ and a $\bmu \in L^2(0,T)^L$ with $\bmu^h \rightarrow \bmu$ for $h\rightarrow\infty$.
	Then we have
	\begin{align}
		|\widehat{j}^h(\bmu^h) - \widehat{j}(\bmu)| \rightarrow 0,
		\qquad\qquad
		\text{ for } h \rightarrow \infty.
	\end{align}
\end{theorem}
\begin{proof}
	Recall the definition of $\hat{j}$ in \eqref{eq:continous_reduced_functional}.
	We calculate
	\begin{align}
		|\widehat{j}^h(\bmu^h) - \widehat{j}(\bmu)| = |\widehat{j}^h(\bmu^h) - \widehat{j}^h(\bmu)| + |\widehat{j}(\bmu) - \widehat{j}^h(\bmu)|.
		\label{eq:convergence_proof}
	\end{align}
	The first summand on the right-hand side in \eqref{eq:convergence_proof} converges to zero due to the continuity of $\widehat{j}^h$.
	The continuity of $\widehat{j}^h$ follows because it is a composition of continuous functions since $\Js$ and the absolute value are continuous.
	The second summand converges to zero due to convergence of quadratures of integral in time \cite[Chapter 5]{Atkinson1991introductionNumAna}, and approximation properties of the discrete expectation value that is known as the \emph{(Weak) Law of Large Numbers} \cite[Theorem 8.2]{GrinsteadSnell1997IntroductionProbability}.
\end{proof}
%

\section{Optimality system}
\label{sec:OptSys}

In this section, we derive the optimality system corresponding to \eqref{eq:OCP_fully_discrete}.
In the following, we consider only a single realization and omit to write the dependence of all quantities to $m$ for the sake of better readability.
An additional reason is that we only need one realization of each path of the $N$ particles in each optimization iteration in our optimization procedure below.

We define the following Lagrange functional $\Lc$, introducing the adjoint processes $\br$ and $\blambda$, as 
\begin{align}
	\Lc(\bz,\bmu,\br,\blambda) \coloneqq
	\frac{1}{N}\sum_{j=1}^{N} \sum_{k = 0}^{N_t^j-1}
	\Lc^{j,k} (\bz,\bmu,\br,\blambda),
\end{align}
with
\begin{multline}
	\Lc^{j,k}(\bz, \bmu, \br, \blambda) 
	\coloneqq
	\sum_{\kappa = N_t^{j,k}}^{N_t^{j,k+1}}
	\Js(\zkappajk,t^\kappa) \Deltatkappa + \frac{\alpha}{2} |\bphi(\zkappajk)^\top\bmu^\kappa|^2 \Deltatkappa
	\\
	-
	\sum_{\kappa = N_t^{j,k}}^{N_t^{j,k+1}-1}
	\Big( \bz_{j}^{\kappa+1} - \zkappajk \ba(\zkappajk,\bmu^\kappa,t^\kappa)\Deltatkappa + \bb(\zkappajk,t^\kappa) \Delta \bB^\kappa_j \Big)^\top \br_{j}^{\kappa+1} 
	\\
	+\Big( \bz_{j}^{N_t^{j,k}} - \bc(\bz_{j}^{N_t^{j,k}-1}) \Big)^\top\blambda_{j}^{N_t^{j,k+1}}
	.
\end{multline}

The optimality system consists of the partial derivatives of the Lagrange functional with respect to its variables set to zero.

Calculating the gradient with respect to $\bz$,
we obtain that $\nabla_{\zkappajk} \Lc \, \delta \zkappajk = 0$ must hold for all admissible directions $\delta \zkappajk$ with $j \in [N]$ and $\kappa = N_t^{j,k},\ldots,N_t^{j,k+1}-1$.
This is equivalent to
\begin{align}
	\br_{j}^\kappa - \br_{j}^{\kappa+1} - \nabla_\bz \ba(\zkappajk,\bmu^\kappa,t^\kappa)^\top \br_{j}^{\kappa+1}
	- \nabla_\bz \bb(\zkappajk,t^\kappa)^\top \br_{j}^{\kappa+1} 
	+ \nabla_\bz j^h(\zkappajk,\bmu)
	= 0 .
	\label{eq:variation_dynamics_adjoint}
\end{align}

We also have to deal with the initial conditions. 
We have that $\nabla_{\bz_{j}^{N_t^{j}}} \Lc \, \delta \bz_{j}^{N_t^{j}} = 0$ and $\nabla_{\bz_{j}^{N_t^{j,k}-1}} \Lc \, \delta \bz_{j}^{N_t^{j,k}-1} = 0$ are equivalent to 
\begin{align*}
	\delta \bz_{j}^{N_t^{j,k}} \Big( \blambda_{j}^{N_t^{j,k+1}} - \br_{j}^{N_t^{j,k}} \Big) = 0
\end{align*}
and
\begin{align*}
	\delta \bz_{j}^{N_t^{j,k}-1} \Big(
	\br_{j}^{N_t^{j,k}-1} - \nabla_\bz \bc(\bz_{j}^{N_t^{j,k}-1})^\top\blambda_{j}^{N_t^{j,k+1}}
	\Big) = 0.
\end{align*}

Using this calculations, and since $\delta\bz_j^\kappa$ was arbitrary, we can conclude that
\begin{align*}
	\blambda_{j}^{N_t^{j,k+1}} = \br_{j}^{N_t^{j,k}}
\end{align*}
and
\begin{align*}
	\br_{j}^{N_t^{j,k}-1} = \nabla_\bz \bc(z_{j}^{N_t^{j,k}-1})^\top\blambda_{j}^{N_t^{j,k+1}}.
\end{align*}
Hence for all $j \in [N]$ and $k \in [N_j]$, the terminal condition for jumps in the adjoint model are given by
\begin{align}
	\br_{j}^{N_t^{j,k}-1} 
	= \nabla_\bz \bc(z_{j}^{N_t^{j,k}-1})^\top \br_{j}^{N_t^{j,k}}.
\end{align}

Furthermore, we get from \eqref{eq:variation_dynamics_adjoint} the dynamics
\begin{align}
	\br_{j}^\kappa = \br_{j}^{\kappa+1} + \nabla_\bz \ba(\zkappajk,\bmu^\kappa,t^\kappa)^\top \br_{j}^{\kappa+1}
	+ \nabla_\bz \bb(\zkappajk,t^\kappa)^\top \br_{j}^{\kappa+1} 
	- \nabla_\bz j^h(\zkappajk,\bmu^\kappa).
	\label{eq:adjoint_dynamics}
\end{align}

Summarizing, we obtain the following adjoint model
\begin{subequations}
	\begin{multline}
		\br_{j}^\kappa = \br_{j}^{\kappa+1} + \nabla_\bz \ba(\zkappajk,\bmu^\kappa,t^\kappa)^\top \br_{j}^{\kappa+1}
		+ \nabla_\bz \bb(\zkappajk,t^\kappa)^\top \br_{j}^{\kappa+1} 
		\\ 
		- \nabla_\bz j^h(\zkappajk)
		\qquad
		\text{ for } \kappa \in [N_t^{j,k},N_t^{j,k+1}-1],
	\end{multline}
	\begin{align}
		\br_{j}^{N_t^{j,k}-1} 
		= \nabla_\bz \bc(z_{j}^{N_t^{j,k}-1})^\top \br_{j}^{N_t^{j,k}}.
	\end{align}
	\label{eq:AdjointModel}
\end{subequations}
Notice that \eqref{eq:AdjointModel} evolves backwards in time.
There are some important differences between the discretize-before-optimize (DBO) approach that we exploit in this work and the optimize-before-discretize approach (OBD).
First of all, we do not have to solve Forward-Backward-SDE (FBSDE) since we do not have to take care of the filtration that evolves forward in time; see, e.g., \cite[Chapter 7, Definition 3.1]{Mao2008SDE} for the definition of solutions to Backward SDEs. 
Since we discretized before we optimized, we cosider for the adjoint equation the same Brownian motion increments and jump-times as for the original system; 
see also \cite{BartschDenkVolkwein2023adjointCalibrationSDE}.
Furthermore, the adjoint jump kernel coincides with one of the model and hence the adjoint jump frequency also coincides with the model jump frequency.
This is different in the OBD case \cite{Bartsch2021MOCOKI}.

Analogously to \cref{thm:ApproximationForwardSDE}, we can formulate a result for the adjoint equation.
For this let $\br^h$ define the continuous representation of a numerical solution for the $j$-th particle in the $m$-th realization as the piecewise constant function as
$\br^h = \{\br^h_{j,m}\}_{j,m=1}^{M,N}$ with
\begin{align}
	\br_{j,m}^h(t) = \sum_{\kappa=0}^{N_t^{j,m}-1} \br_{j,m}^\kappa \chi_{[t^\kappa,t^{\kappa+1}]}(t)
	\qquad\qquad t \in [0,T].
	\label{eq:piecewise_constant_approx_adjoint}
\end{align}

\begin{theorem}
	Let \cref{asmpt:existence_uniqueness_SDE,asmpt:functional} hold.
	Let $\br^h$ be the piecewise constant approximation given in \eqref{eq:piecewise_constant_approx_adjoint}.
	Then there exists for any $\bmu \in L^2(0,T)^N$ a constant $C>0$ and $N_t^* \in \Nb$ such that for all $N_t \geq N_t^*$
	\begin{align}
		\Eb \sup_{t \in [0,T]} [|\br^h (t) - \br(t)|^2] \leq C \Delta t(1 + \Eb[|\br(0)|^2] ).
	\end{align}
\end{theorem}

\bigskip

The next step is to derive the reduced gradient $\nabla_{\bmu} \widehat{j}^h$ of $\widehat{j}^h$ defined in \eqref{eq:discrete_reduced_functional}.
For this, we take the directional derivative of $\Lc$ with respect to $\bmu^\kappa$ for $\kappa \in [N_t^u]$ and calculate for the arbitrary direction $\delta \bmu^\kappa$
\begin{align}
	\nabla_{\bmu^k} \Lc \, \delta \bmu^\kappa
	= \frac{\Delta t}{N}\sum_{j=1}^N  \delta \bmu^\kappa \left\lbrace
	\alpha \, \bphi(\bz^h_j(t^\kappa))\bphi(\bz^h_j(t^\kappa))^\top \bmu^\kappa - \Psi(\bz^{h}_j(t^\kappa))^\top \br_j^h(t^\kappa)
	\right\rbrace.
	\label{eq:Gradient}
\end{align}

Hence, the reduced gradient is given by
\begin{align}
	\nabla_{\bmu^\kappa} \widehat{j}^h(\bmu^\kappa) = \frac{\Delta t}{N}\sum_{j=1}^N 
	\alpha \, \bphi(\bz^h_j(t^\kappa))\bphi(\bz^h_j(t^\kappa))^\top \bmu^\kappa - \Psi(\bz^{h}_j(t^\kappa))^\top \br_j^h(t^\kappa),
	\label{eq:definition_Gradient}
\end{align}
where $\bz$ and $\br$ are solutions to \eqref{eq:model_approximation} and \eqref{eq:AdjointModel}, respectively, corresponding to $\bmu$.
Hence, we have that at optimality it must hold for all $\kappa \in [N_t^u]$ that
\begin{align}
	\nabla_{\bmu^\kappa} \widehat{j}^h(\bmu^\kappa) = 0
\end{align}
which is equivalent to
\begin{align}
	\sum_{j=1}^N 
	\alpha \, \bphi(\bz^h_j(t^\kappa))\bphi(\bz^h_j(t^\kappa))^\top \bmu^\kappa 
	=
	\sum_{j=1}^N  \Psi(\bz^{h}_j(t^\kappa))^\top \br_j^h(t^\kappa)
	\qquad\qquad
	\forall \kappa \in [N_t^u].
	\label{eq:optimality_condition}
\end{align}

Notice that \eqref{eq:optimality_condition} is a high-dimensional nonlinear system with respect to $\bmu^\kappa$.

\section{Numerical implementation and optimization procedure}
\label{sec:NumImpl}

In this section, we explain our strategy to solve the model and adjoint problem and the procedure to solve the full optimization problem \eqref{eq:OCP_fully_discrete}.
We use Monte Carlo strategies to solve the arising equations.
In principle, these methods are meshless.
However, to define our shape functions $\bphi$, we shall introduce a computational domain and a discrete phase space.
We want to point out, that we do not specify boundary conditions, as this computational domain is needed only to define the number and the centers of the shape functions.

In \cref{sec:Discretization}, we introduce the computational domain and define the shape functions in \cref{sec:Functions}.
In \cref{sec:Jumps}, we explain our procedure to handle the jump processes before we discuss our strategy to solve the model and adjoint equation in \cref{sec:FreeFlight} and \cref{sec:Adjoint}, respectively.
In \cref{sec:OptiProc}, we discuss our optimization strategy.

\subsection{Discretization}
\label{sec:Discretization}

We consider our two-dimensional computational domain as $\Omega_x \times \Omega_v \coloneqq (-\xmax,\xmax) \times (-\vmax,\vmax)$, with given $\xmax,\vmax>0$.
We use a standard finite-volume discretization in phase space and define the discrete phase space $\Omega_{\Delta x, \Delta v} \coloneqq \Omega_{\Delta x} \times \Omega_{\Delta v}$ as follows; see also \cite{Bartsch2021MOCOKI, BarthHerbinOhlberger2018FiniteVolume}.
We choose a partition of $\Omega_x$ and $\Omega_v$ of equally-spaced, non-overlapping square cells with side length $\Delta v = 2\vmax/N_v$ where $N_v \geq 2$ and $\Delta x = 2\xmax/N_x$ with $N_x\geq2$, respectively. 
On this partition, we consider a cell-centered representation as follows:
\begin{align*}
	\Omega_{\Delta x} 
	\coloneqq \left\lbrace \; x^i \in \Omega_x \;\big\vert\;  i \in [N_x] \; \right\rbrace, 
	\qquad\qquad
	x^i \coloneqq (i-\nicefrac{1}{2})\Delta x - \xmax
\end{align*}
and
\begin{align*}
	\Omega_{\Delta v} \coloneqq \left\lbrace \; v^l \in \Omega_v \;\big\vert\; l \in [N_v] \; \right\rbrace,
	\qquad\qquad
	v^l \coloneqq \left(l-\nicefrac{1}{2}\right)\Delta v - \vmax.
\end{align*}
Notice that the computational phase space is centered at $(0,0)$ by this choice.

In order to discretize the control $\bmu_\ell$, $\ell \in [L]$, with $L=N_x \, N_v$, we use the time discretization given in \eqref{eq:Deterministic_timestep}. 
We identify the control and the solution $\bz$ as the constant approximation between the grid points corresponding to this discretization of $[0,T]$; cf. \eqref{eq:piecewise_constant_approx} and \eqref{eq:piecewise_constant_control}.

\subsection{Shape functions}
\label{sec:Functions}

For our implementation of the shape functions $\varphi_\ell$, $\ell \in [L]$, we consider \emph{radial-basis functions} (RBF) \cite{Buhmann2003RadialBasisFunctions}.
More specifically, we choose a certain form of a bump function with center $x_c$ and shape parameter $\varepsilon_\varphi>0$ given by
\begin{align}
	\varphi(x,x_c,\varepsilon_\varphi) \coloneqq
	\begin{cases}
		\exp\left(-  \frac{1}{1-(\varepsilon_\varphi |x-x_c|)^2}\right) &\text{ for } |x-x_c|<\frac{1}{\varepsilon_\varphi} ,\\
		0 &\text{ else}.
	\end{cases}
	\label{eq:Gaussian_RBF}
\end{align}
The functions $\phi^x_\ell,\phi^v_\ell$, $\ell \in [L]$, are defined using $\varphi$ by
\begin{subequations}
	\begin{align}
		&\phi^x_\ell: \Rb \rightarrow \Rb,
		&&\phi^x_\ell(x) = \varphi(x,x^\ell,\varepsilon_\varphi),
		\\
		&\phi^v_\ell: \Rb \rightarrow \Rb,
		&&\phi^v_\ell(v) = \varphi(v,v^\ell,\varepsilon_\varphi).
	\end{align}
\end{subequations}

There are results for the convergence of the approximation of $L^2$ functions using RBF with compact support present in the literature; see, e.g., \cite[Theorem 6.7]{Buhmann2003RadialBasisFunctions}.

The derivative of the shape functions is given by:
\begin{align*}
	\varphi_x(x,x_c,\varepsilon_\varphi) = 
	\begin{cases}
		\exp\left(-  \frac{1}{1-(\varepsilon_\varphi |x-x_c|)^2}\right)
		\frac{-2\varepsilon_\varphi |x-x_c|}{(1-(\varepsilon_\varphi |x-x_c|)^2)^2}
		&\text{ for } |x-x_c|<\frac{1}{\varepsilon_\varphi}  ,
		\\
		0 &\text{ else}.
	\end{cases}
\end{align*}

In \cref{fig:vis_Bump_RBF}, we visualize the RBF function together with its derivative.
We point out, that both of them are defined on whole $\Rb$ and are compactly supported.
Notice that we choose the same shape function for position and velocity for the sake of simplicity. 
However, it is in principle possible to choose different functions for each component. 

\begin{figure}
	\centering
	\includegraphics[width=0.35\textwidth]{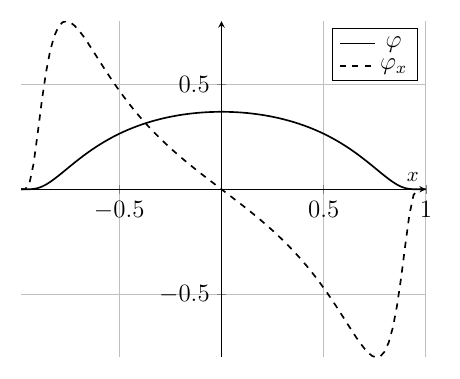}
	\caption{A radial basis function $\varphi$ and its derivative $\varphi_x$.}
	\label{fig:vis_Bump_RBF}
\end{figure}

\subsection{Jump process}
\label{sec:Jumps}

For our implementation, we consider the jump part $\bc$ to be related to the Keilson-Storer kernel introduced in \cite{KeilsonStorer1952KS}.
For a pre-jump velocity $v \in \Rb$ and post-jump velocity $w \in \Rb$, it is given by
\begin{align}
	A(v, w) = \Gamma \sqrt{\frac{\beta}{\pi}} \exp (-\beta |w-\gamma v|^2) ,
	\label{eq:Keilson-Storer_kernel}
\end{align}
with parameters $\gamma \in[-1,1]$ and $\beta >0$.

With this kernel, many physical phenomena can be modeled including Brownian motion, telegraphic noise, and the famous Bhatnagar--Gross--Krook (BGK) collision operator \cite{BGK1954}. 
It was also used in previous work \cite{Bartsch2021MOCOKI,BartschBorzi2024MOCOKIFeedbackStabilization} by the authors.
For this work, we define for a position $x \in \Rb$ and velocity $v \in \Rb$
\begin{align}
	\bc(x,v) = (x,\gamma v + \varsigma_\beta)^\top,
	\label{eq:jump_law}
\end{align}
where $\varsigma_\beta$ are precalculated random numbers obeying $\mathcal{N}(0,(2\beta)^{-1})$. We denote by $\mathcal{N}(\xi,\varpi)$ the Gaussian distribution with mean $\xi$ and variance $\varpi$.
The parameters $\gamma$ and $\beta$ are given in the definition of the Keilson-Storer kernel \eqref{eq:Keilson-Storer_kernel}.

Furthermore, we need to sample the jump times $T_{j,k}$. 
For this, we use the fact that the mean free time is given as the reciprocate of jump frequency of the Keilson-Storer kernel defined as
\begin{align}
	\sigma  =  \int_{\Rb} A(v,w) \rmd v = \sqrt{\frac{\beta}{\pi}}.
	\label{eq:collision_frequency}
\end{align}
In order to determine the time instances at which a particle undergoes a velocity transition due to jump using $\sigma$, one can follow the procedure described in, e.g., \cite{Jacoboni1983,Bartsch2020OCPKS}.

If $\sigma$ is the
jump frequency, then $\sigma \rmd t$ is the probability that a particle has a jump during the time $\rmd t$. 
Now, assuming that a particle has a jump at time $t$, the probability that it will be subject to another jump at time $t + \delta t$ dt is computed according to a Poisson distribution given by
\begin{align}
	\exp\left( - \int_{t}^{t + \delta t} \sigma \rmd t' \right)
	= \exp \left(-\delta t \, \sigma \right).
	\label{eq:collision_time_Poisson}
\end{align}
The formula in \eqref{eq:collision_time_Poisson} is the probability distribution of the distance between two events of a Poisson process.
Following a standard approach \cite{Jacoboni1983} and using a uniformly distributed random number $\nu \in (0,1)$, one obtains the following rule
\begin{align}
	\delta t = - \sigma^{-1} \log (\nu).
\end{align}
Notice that we do not have an adjoint jump frequency, like we have it in \cite{Bartsch2020OCPKS}.

\subsection{Model equation}
\label{sec:FreeFlight}

For our numerical examples in \cref{sec:NumExp}, we consider in the following structure the coefficient functions
\begin{subequations}
	\begin{align}
		\ba(\bz,u,t) = \mathrm{e} \otimes   \binom{0}{u(\bz,t)}
		+ H\bz + h(\bz,t),
		\qquad
		H \coloneqq \begin{pmatrix}
			0 & 1 \\
			-\eta & 0\\
		\end{pmatrix} \otimes I_N,
	\end{align}
	where $h$ is a given smooth nonlinear function such that $\ba$ fulfills \cref{asmpt:existence_uniqueness_SDE}, $I_N \in \Rb^{N\times N}$ is the $N$-dimensional identity matrix and
	\begin{align}
		\bb(\bz,t) = \begin{pmatrix}
			\mathrm{b}_1  & 0 \\
			0 & \mathrm{b}_2
		\end{pmatrix} \otimes I_N,	
		\qquad\qquad
		\mathrm{b}_1,\mathrm{b}_2>0.
	\end{align}
	\label{eq:coefficients}
\end{subequations}

During the free flight, we consider Euler's method for integration using stepsize $\Delta t$ that is smaller than the difference of two jump times.
The Euler-Maruyama method is given by
\begin{subequations}
	\begin{align}
		&\bx^{\kappa+1}_j = \bx^\kappa_j +  \bv^\kappa_j \,\Delta t  +\mathrm{b}_1 \, \Delta \bB^{x,\kappa}_j
		&\bx^{N_t^{j,k}}_j = \bx^{N_t^{j,k}-1}_j,
		\\
		&\bv^{\kappa+1}_j =  \bv^\kappa_j + \Big( u(\bz^\kappa_j,t^\kappa) - \eta \, \bx^\kappa_j \Big) \, \Delta t  +\mathrm{b}_2 \,\Delta \bB^{v,\kappa}_j
		&\bv^{N_t^{j,k}}_j = \gamma \bv^{N_t^{j,k}-1}_j + \varsigma_\beta,
	\end{align}
	\label{eq:SDE_splitted_drift_Euler}%
\end{subequations}
where $\Delta \bB^{x,\kappa}_j$ and $\Delta \bB^{v,\kappa}_j$ for $j \in [N]$ and $\kappa \in [N_t]$ are precomputed Gaussian samples for position and velocity with zero mean and standard deviation $\sqrt{\Delta t^\kappa}$; see \cite{Maruyama1955SDE}.
The whole method for solving the forward model is summarized in \cref{algo:particle_solver,algo:model_solver}.

\subsection{Adjoint equation}
\label{sec:Adjoint}

Recall that the adjoint equation \eqref{eq:AdjointModel} evolves backward.
We introduce the adjoint position $\bq$ and adjoint velocity $\bp$ for the adjoint phase space coordinate $\br = (\bq,\bp)$.
The adjoint equations for the coefficients fixed in \eqref{eq:coefficients} are given by
\begin{subequations}
	\begin{align}
		&\bq^\kappa_j = \bq^{\kappa+1}_{j} + 
		\Delta t \Big( \left( u_x(\zkappajk) - \eta\right) \bp^{\kappa+1}_j  
		- \nabla_\bx j^h(\zkappajk,u(\zkappajk)) \Big),
		\\ 
		&\bp^\kappa_j = \bp^{\kappa+1}_j 
		+ \Delta t \Big( \bq^{\kappa+1}_j + u_v(\bz^{\kappa}_j) \bp^{\kappa+1}_j 
		- \nabla_\bv j^h(\zkappajk,u(\zkappajk))\Big),
	\end{align}
	\label{eq:adjoint_eq_approximation}
\end{subequations}
with terminal conditions
\begin{align}
	\bq_{j,T_{N_{j,k}}} = -\nabla_\bx j^h\left(\bz_{j,T_{N_{j,k}}},u(\bz_{j,T_{N_{j,k}}})\right),
	\qquad
	\bp_{{j,T_{N_{j,k}}}} = -\nabla_\bv j^h\left(\bz_{j,T_{N_{j,k}}},u(\bz_{j,T_{N_{j,k}}})\right).
	\label{eq:terminal_condition_adjoint}
\end{align}
The derivatives of $u$ are given by
\begin{align*}
	u_x(x,v,t) = \sum \bmu_\ell \partial_x \bphi_\ell^x(x)  \bphi_\ell^v(v),
	\qquad\qquad
	u_v(x,v,t) = \sum \mu_\ell  \bphi_\ell^x(x) \partial_v \bphi_\ell^v(v).
	\label{eq:derivatives_control} 
\end{align*}
The derivatives of the functional are given by
\begin{align}
	\nabla_\bx j^h(\bz_j,u) = \nabla_\bx \Js (\bz_j) + \alpha \, u(\bz_j) u_x(\bz_j),
	\qquad
	\nabla_\bv j^h(\bz_j,u) = \nabla_v \Js (\bz_j) + \alpha \, u(\bz_j) u_v(\bz_j).
\end{align}

\subsection{Optimization procedure}
\label{sec:OptiProc}

In this section, we describe our numerical strategy to solve the discrete optimization problem \eqref{eq:OCP_fully_discrete}.
We exploit a stochastic gradient method (SGD) with linesearch \cite{Vaswani2019painlessSGDLinesearch}.
To assemble the gradient in every iteration, we need to integrate the equations of motion for the model \eqref{eq:SDE_splitted_drift_Euler} and adjoint equation \eqref{eq:adjoint_eq_approximation}.
In \cref{algo:particle_solver}, we summarize this procedure taking into account the jumps.

\begin{algorithm}
	\caption{Integration of the equations of motion}
	\label{algo:particle_solver}
	\small
	\begin{algorithmic}[1]
		\Require Initial position $\bx^0_{j}$ and velocity $\bv^0_{j}$, number of jump times $N_j$ of the $j$-th particle
		\State Set $t^\kappa \gets 0$, $t\gets 0$, $\bar{t} \gets 0$, $\kappa \gets 0$
		\While{$t < T$}
		\If{$t^{\kappa+1}>T^{j,k}$} \Comment{\emph{jump}}
		\State $\bar t \gets T^{j,k}-t$
		\State Update position and velocity using Euler's method according to \eqref{eq:SDE_splitted_drift_Euler} (respectively \eqref{eq:adjoint_eq_approximation}) using $\bar{t}$ as stepsize in time
		\State Calculate the initial condition of the next time interval using $\bc(\bv_{j,k})$ in \eqref{eq:jump_law} (respectively $\nabla\bc$)
		\State $t \gets T^{j,k}$
		\Else \Comment{\emph{no jump}}
		\State $\bar t \gets t^{k+1} - t$
		\State Update position and velocity using Euler's method according to \eqref{eq:SDE_splitted_drift_Euler} (respectively \eqref{eq:adjoint_eq_approximation}) using $\Delta t$ as stepsize in time $\bar t$.
		\State Set $t \gets t^{\kappa+1}$
		\EndIf			
		\State $\kappa \gets \kappa+1$
		\EndWhile
	\end{algorithmic}
\end{algorithm}

In \cref{algo:model_solver}, we summarize the procedure to solve the state equation.
Notice that since the particles evolve independently from each other, one can heavily parallelize this method.
Furthermore, notice that we generate jump times for each particle a priori.
This jump times are then also used in the solver for the adjoint equation.

\begin{algorithm}
	\caption{Model solver}
	\label{algo:model_solver}
	\small
	\begin{algorithmic}[1]
		\Require Parameters $\gamma,\beta>0$, initial particles $\bz^0$
		\For{each particle $j=1$ \textbf{to} $N$} \Comment{\emph{particles evolve independently from each other}}
		\State $T_{j,0} \gets 0$, $\tilde{k} \gets 0$
		\While{$T_{j,\tilde{k}} < T$} \Comment{\emph{Sample jump times}}
		\State Set $\tilde{k} \gets \tilde{k}+1$
		\State Sample new jump time $\delta t_{j,\tilde{k}}$ using
		\eqref{eq:collision_time_Poisson} and set $T_{j,\tilde{k}} \gets T_{j,\tilde{k}-1} + \delta t_{j,\tilde{k}}$ 
		\EndWhile\Comment{\emph{Create $\cup_{k=1}^{N_j} (T_{j.k-1}, T_{j,k}] = (0,T]$}}
		\State{Set $N_j \gets \tilde{k}-1$}
		\State Sample $\varsigma^k_\beta$, $k \in [N_j]$ obeying $\mathcal{N}(0,(2\beta)^{-1})$ for each jump time $T_{j,k}$
		\State Integrate the equations of motion of the particles using \cref{algo:particle_solver}
		\EndFor \\
		\Return 
	\end{algorithmic}
\end{algorithm}

The method to calculate the reduced gradient \eqref{eq:definition_Gradient} is presented in \cref{algo:CalculateGradient}. 
This procedure is executed iteratively in the complete optimization algorithm summarized  in \cref{algo:OptAlgo}.
In each iteration, we sample a new set of initial conditions for the particles obeying the initial distribution $\mathrm{\bz}$. 
Then we solve the model equation where we generate the jump times and Brownian motion increments.
These are then also used in the solver for the adjoint equation.
The resulting (discrete) trajectories for all particles are taken in to account to assemble the reduced gradient for this iteration.

\begin{algorithm}
	\caption{Calculate the gradient}
	\label{algo:CalculateGradient}
	\small
	\begin{algorithmic}[1]
		\Require current control iterate $\bmu^n=(\bmu^n,\ldots,\bmu^n_{N_t})$, initial distribution $\mathring{\bz} = (\mathring{\bx},\mathring{\bv})$, desired trajectory $\zd(t) = (\xd(t),\vd(t))$
		\State Sample initial condition of $N$ particles obeying $\mathring{\bz}$
		\State Solve the model system \eqref{eq:model_approximation} using \cref{algo:model_solver}
		\State Calculate terminal condition for adjoint equation according to \eqref{eq:terminal_condition_adjoint}
		\State Solve adjoint system \eqref{eq:AdjointModel} using the timesteps generated in \cref{algo:model_solver} and the integration analog to the one in \cref{algo:particle_solver}
		\State Calculate the reduced gradient $\nabla_{\bmu} \widehat{j}^h$ according to \eqref{eq:definition_Gradient}.
		\\ \Return $\nabla_{\bmu} \widehat{j}^h$
	\end{algorithmic}
\end{algorithm}

The resulting gradient is then used as a stepdirection as usual in the stochastic gradient descent; see, e.g, \cite{Higham2019DeepLearning}.
Additionally, we consider the stochastic version of the classical Armijo-linesearch to determine a stepsize $\zeta_n>0$
\begin{align}
	\widehat{j}^h \left(\bmu^n - \zeta_n \nabla \widehat{j}^h (\bmu_n)\right) 
	\leq 
	\widehat{j}^h (\bmu_n) - c \, \zeta_n \|\nabla   \widehat{j}^h (\bmu_n)\|^2,
	\label{eq:Stochastic_Armijo}
\end{align}
with a hyperparameter $c>0$; see \cite{Vaswani2019painlessSGDLinesearch,Dvinskikh2020LinesearchStochastic}.
The algorithm terminates if the difference between two subsequent control iterates is closer than a given tolerance or if the maximum iteration depth $n_{\mathrm{max}}$ is reached.
The whole strategy is summarized in \cref{algo:OptAlgo}.

\begin{algorithm}
	\caption{Gradient descent scheme}
	\label{algo:OptAlgo}
	\small
	\begin{algorithmic}[1]
		\Require desired state $\zd(t) = (\xd(t),\vd(t))$, 
		initial guess of the control $\bmu^0=(\bmu^0,\ldots,\bmu^0_{N_t})$, initial distribution $\mathring{\bz}=(\mathring{\bx},\mathring{\bv})$, tolerance $\tol>0$, maximum iteration depth $n_{\max}$
		\State Set $n \gets 0$ and initialize $\mathrm{E} \gg tol$
		\While{$\mathrm{E}  > \tol$ \textbf{and} $n<n_{\max}$}
		\State Compute reduced gradient $\bh^n$ using \cref{algo:CalculateGradient} \Comment{\emph{Sampling of random data included here}}
		\State Determine the step-size $\zeta_n$ along $\bh^n$ satisfying
		\eqref{eq:Stochastic_Armijo}
		\State Update control: $\bmu^{n+1} \gets \bmu^n + \zeta_n \, \bh^n$
		\State $\mathrm{E} \gets \|\bmu^{n+1}-\bmu^n\|_2$
		\State Set $n \gets n+1$
		\EndWhile
		\\ \Return $\bU^\ell$
	\end{algorithmic}
\end{algorithm}

\begin{figure}
	\footnotesize
	\centering
	
	\tikzset{every picture/.style={line width=0.75pt}}
	\begin{tikzpicture}[x=0.75pt,y=0.75pt,yscale=-0.8,xscale=0.8]
		
		\draw  [fill={rgb, 255:red, 0; green, 0; blue, 0 }  ,fill opacity=0.14 ] (173,179) -- (408,179) -- (447.33,218.33) -- (447.33,403) -- (173,403) -- cycle ;
		\draw   (30,11.4) .. controls (30,6.76) and (33.76,3) .. (38.4,3) -- (376.93,3) .. controls (381.57,3) and (385.33,6.76) .. (385.33,11.4) -- (385.33,42.6) .. controls (385.33,47.24) and (381.57,51) .. (376.93,51) -- (38.4,51) .. controls (33.76,51) and (30,47.24) .. (30,42.6) -- cycle ;
		\draw   (212.83,67) -- (266.33,111) -- (212.83,155) -- (159.33,111) -- cycle ;
		\draw   (176,213) .. controls (176,209.13) and (179.13,206) .. (183,206) -- (424.33,206) .. controls (428.2,206) and (431.33,209.13) .. (431.33,213) -- (431.33,234) .. controls (431.33,237.87) and (428.2,241) .. (424.33,241) -- (183,241) .. controls (179.13,241) and (176,237.87) .. (176,234) -- cycle ;
		\draw   (244,284.2) .. controls (244,280.22) and (247.22,277) .. (251.2,277) -- (413.13,277) .. controls (417.11,277) and (420.33,280.22) .. (420.33,284.2) -- (420.33,305.8) .. controls (420.33,309.78) and (417.11,313) .. (413.13,313) -- (251.2,313) .. controls (247.22,313) and (244,309.78) .. (244,305.8) -- cycle ;
		\draw   (198,358) .. controls (198,355.79) and (199.79,354) .. (202,354) -- (392,354) .. controls (394.21,354) and (396,355.79) .. (396,358) -- (396,370) .. controls (396,372.21) and (394.21,374) .. (392,374) -- (202,374) .. controls (199.79,374) and (198,372.21) .. (198,370) -- cycle ;
		\draw   (2,353.4) .. controls (2,349.31) and (5.31,346) .. (9.4,346) -- (133.6,346) .. controls (137.69,346) and (141,349.31) .. (141,353.4) -- (141,375.6) .. controls (141,379.69) and (137.69,383) .. (133.6,383) -- (9.4,383) .. controls (5.31,383) and (2,379.69) .. (2,375.6) -- cycle ;
		\draw   (1,215.2) .. controls (1,210.67) and (4.67,207) .. (9.2,207) -- (155.8,207) .. controls (160.33,207) and (164,210.67) .. (164,215.2) -- (164,239.8) .. controls (164,244.33) and (160.33,248) .. (155.8,248) -- (9.2,248) .. controls (4.67,248) and (1,244.33) .. (1,239.8) -- cycle ;
		\draw   (35,96.2) .. controls (35,91.67) and (38.67,88) .. (43.2,88) -- (105.8,88) .. controls (110.33,88) and (114,91.67) .. (114,96.2) -- (114,120.8) .. controls (114,125.33) and (110.33,129) .. (105.8,129) -- (43.2,129) .. controls (38.67,129) and (35,125.33) .. (35,120.8) -- cycle ;
		\draw    (69.33,346) -- (70.31,252) ;
		\draw [shift={(70.33,250)}, rotate = 90.6] [color={rgb, 255:red, 0; green, 0; blue, 0 }  ][line width=0.75]    (10.93,-3.29) .. controls (6.95,-1.4) and (3.31,-0.3) .. (0,0) .. controls (3.31,0.3) and (6.95,1.4) .. (10.93,3.29)   ;
		\draw    (112.33,110.5) -- (124.34,110.63) -- (157.33,110.98) ;
		\draw [shift={(159.33,111)}, rotate = 180.61] [color={rgb, 255:red, 0; green, 0; blue, 0 }  ][line width=0.75]    (10.93,-3.29) .. controls (6.95,-1.4) and (3.31,-0.3) .. (0,0) .. controls (3.31,0.3) and (6.95,1.4) .. (10.93,3.29)   ;
		\draw    (331,253) -- (221.33,252) ;
		\draw    (331,253) -- (331.46,263.95) -- (331.92,275) ;
		\draw [shift={(332,277)}, rotate = 267.61] [color={rgb, 255:red, 0; green, 0; blue, 0 }  ][line width=0.75]    (10.93,-3.29) .. controls (6.95,-1.4) and (3.31,-0.3) .. (0,0) .. controls (3.31,0.3) and (6.95,1.4) .. (10.93,3.29)   ;
		\draw    (219.33,332) -- (335,333) ;
		\draw    (198,364) -- (144,364) ;
		\draw [shift={(142,364)}, rotate = 360] [color={rgb, 255:red, 0; green, 0; blue, 0 }  ][line width=0.75]    (10.93,-3.29) .. controls (6.95,-1.4) and (3.31,-0.3) .. (0,0) .. controls (3.31,0.3) and (6.95,1.4) .. (10.93,3.29)   ;
		\draw    (212.67,51) -- (212.81,65) ;
		\draw [shift={(212.83,67)}, rotate = 269.4] [color={rgb, 255:red, 0; green, 0; blue, 0 }  ][line width=0.75]    (10.93,-3.29) .. controls (6.95,-1.4) and (3.31,-0.3) .. (0,0) .. controls (3.31,0.3) and (6.95,1.4) .. (10.93,3.29)   ;
		\draw    (212.83,155) -- (213.6,177) ;
		\draw [shift={(213.67,179)}, rotate = 268.01] [color={rgb, 255:red, 0; green, 0; blue, 0 }  ][line width=0.75]    (10.93,-3.29) .. controls (6.95,-1.4) and (3.31,-0.3) .. (0,0) .. controls (3.31,0.3) and (6.95,1.4) .. (10.93,3.29)   ;
		\draw    (266.33,111) -- (330.33,111) ;
		\draw [shift={(332.33,111)}, rotate = 180] [color={rgb, 255:red, 0; green, 0; blue, 0 }  ][line width=0.75]    (10.93,-3.29) .. controls (6.95,-1.4) and (3.31,-0.3) .. (0,0) .. controls (3.31,0.3) and (6.95,1.4) .. (10.93,3.29)   ;
		\draw    (221.33,243) -- (221.33,351) ;
		\draw [shift={(221.33,353)}, rotate = 270] [color={rgb, 255:red, 0; green, 0; blue, 0 }  ][line width=0.75]    (10.93,-3.29) .. controls (6.95,-1.4) and (3.31,-0.3) .. (0,0) .. controls (3.31,0.3) and (6.95,1.4) .. (10.93,3.29)   ;
		\draw    (70.33,206) -- (70.98,131) ;
		\draw [shift={(71,129)}, rotate = 90.5] [color={rgb, 255:red, 0; green, 0; blue, 0 }  ][line width=0.75]    (10.93,-3.29) .. controls (6.95,-1.4) and (3.31,-0.3) .. (0,0) .. controls (3.31,0.3) and (6.95,1.4) .. (10.93,3.29)   ;
		\draw    (334.33,314) -- (335,333) ;
		
		\draw (38,3) node [anchor=north west][inner sep=0.75pt]  [font=\footnotesize] [align=left] {Choose initial guess $\displaystyle u^{0}$, desired trajectory $\zd$, \\ hyperparameters, \\and initialize $\displaystyle \mathrm{E} \gg \mathrm{tol}$, $\displaystyle n\leftarrow 0$};
		\draw (175,95) node [anchor=north west][inner sep=0.75pt]  [font=\footnotesize] [align=left] {$\displaystyle  \begin{array}{{>{\displaystyle}l}}
				\mathrm{E} \gg \mathrm{tol}\\
				n< n_{\max}
			\end{array}$};
		\draw (279,94) node [anchor=north west][inner sep=0.75pt]  [font=\footnotesize] [align=left] {\textbf{false}};
		\draw (226.5,154) node [anchor=north west][inner sep=0.75pt]  [font=\footnotesize] [align=left] {\textbf{true}};
		\draw (185,209) node [anchor=north west][inner sep=0.75pt]  [font=\footnotesize] [align=left] {Solve forward model with \\
			\cref{algo:model_solver} and \cref{algo:particle_solver}};
		\draw (250,280) node [anchor=north west][inner sep=0.75pt]  [font=\footnotesize] [align=left] {Solve adjoint model analog \\to solving forward model};
		\draw (204,357) node [anchor=north west][inner sep=0.75pt]  [font=\footnotesize] [align=left] {Calculate gradient (cf. \eqref{eq:definition_Gradient})};
		\draw (11.4,349) node [anchor=north west][inner sep=0.75pt]  [font=\footnotesize] [align=left] {perform Armijo \\lineserach \eqref{eq:Stochastic_Armijo}};
		\draw (19,209) node [anchor=north west][inner sep=0.75pt]  [font=\footnotesize] [align=left] {Update Control\\$\displaystyle \bmu ^{n+1} =\bmu ^{n} + \zeta _{n} \ \boldsymbol{g}^{n}$};
		\draw (41,99) node [anchor=north west][inner sep=0.75pt]  [font=\footnotesize] [align=left] {$\displaystyle n=n+1$};
		\draw (336,105) node [anchor=north west][inner sep=0.75pt]  [font=\footnotesize] [align=left] {stop};
		\draw (202,271) node [anchor=north west][inner sep=0.75pt]  [font=\footnotesize] [align=left] {$\bz_{j}^{h}$};
		\draw (351,325) node [anchor=north west][inner sep=0.75pt]  [font=\footnotesize] [align=left] {$\br_{j}^{h}$};
		\draw (81,254) node [anchor=north west][inner sep=0.75pt]   [align=left] {$\zeta _{n}$};
		\draw (146,344) node [anchor=north west][inner sep=0.75pt]  [font=\footnotesize] [align=left] {$\boldsymbol{g}^{n}$};
		\draw (175,181) node [anchor=north west][inner sep=0.75pt]  [font=\footnotesize] [align=left] {\textbf{\cref{algo:CalculateGradient}}};	
	\end{tikzpicture}
	\caption{UML flowchart of the optimization procedure \cref{algo:OptAlgo}.}
	\label{fig:UML}
\end{figure}
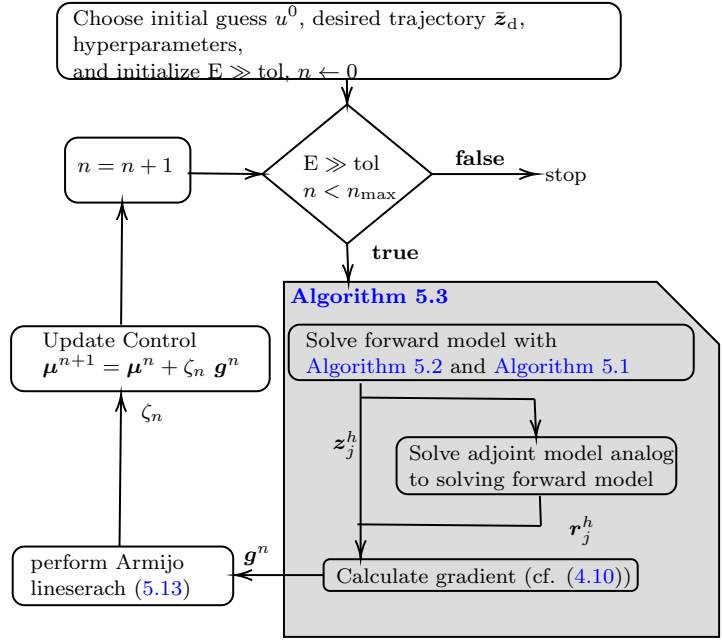

\section{Numerical experiments}
\label{sec:NumExp}

In this section, we show the results several numerical experiments in order to validate our optimization procedure.
For all of them, we fix the numerical computational phase space having the bounds $\xmax=2$ and $\vmax=2$. 
For the basis functions $\varphi$ in \eqref{eq:Gaussian_RBF}, we take the grid-points of $\Omega_{\Delta x} \times \Omega_{\Delta v}$ as centers of the radial basis functions.
Hence, we have $L=N_x \, N_v$ basis functions.
We iterate through all center points $(x^i,v^l)$ while assembling $u$ in \eqref{eq:control_structure}.
For the final time, we choose $T=1$.

As tracking cost, we consider for $\bz = (x,v) \in \Rb^2$
\begin{align}
	\Js(\bz) = -\exp\left(-\frac{1}{2\sigma_\Js^2}(|x-\xd(t)|^2+|v-\vd(t)|^2)\right)
	\label{eq:tracking_cost}
\end{align}
with a desired (mean) phase space trajectory $\zd(t) = (\xd(t),\vd(t))$.
Notice that $\Js$ is bounded from above and below and smooth and hence satisfies \cref{asmpt:functional}.
Its partial derivatives are then given by
\begin{align*}
	\begin{split}
		\nabla_x \Js(\bz) = \sigma_\Js^{-2}(x-\xd)\exp\left(-\frac{1}{2\sigma_\Js^2}(|x-\xd|^2+|v-\vd|^2)\right),
		\\
		\nabla_v \Js(\bz) = \sigma_\Js^{-2}(v-\vd)\exp\left(-\frac{1}{2\sigma_\Js^2}(|x-\xd|^2+|v-\vd|^2)\right).
	\end{split}
\end{align*} 

In all experiments, we use $\beta = 10$ and $\gamma=0.9$.
Furthermore, for the temporal discretization, we choose $N_t = 50$, $\Delta t = 0.1$ which leads to $T=5$.
Furthermore, we consider $N_x = N_v = 10$, $\Delta x = \Delta v = 0.4$ and use $\varepsilon_\varphi=0.5$ and consider $N=2\cdot10^3$ particles.

We start with a test case where we want to center all particles in the middle of the computational domain given a Gaussian distribution as initial configuration and continue with the case where we have a uniform distribution in a subdomain of the computational domain and also want to have the particles centered in the middle (cf. \cref{sec:Centering}).
Then, in \cref{sec:Stabilization}, we average the control found in this case in time and apply it to a random initial configuration that also has particles outside the previous subdomain.
The next test case is to follow a (non-smooth) trajectory in time (cf. \cref{sec:TimeDependentTrajectory}). 
After this, we test our framework with a system of particles including coupling in \cref{sec:Coupling}.
Finally, we present in \cref{sec:AdjointIndependent} a strategy that avoids storing the forward trajectories and test it for the setting of the previous case.

\subsection{Centering of the particles}
\label{sec:Centering}
For our first example, we choose $\zd(t) = (0,0)$. 
Moreover, we take $\eta=1$ in \eqref{eq:coefficients}.
We start first with a normal distributed initial condition $\mathring{z} \sim \mathcal{N}(0.75,0.75)^\top,0.01 \, I_2)$.
In \cref{fig:Centering_NormalDistribution}, we plot the result of this numerical experiment.
In \cref{fig:Centering_Normal_InitialFinal}, we plot the initial configuration (gray) and the one obtained in the final timestep using our optimized control.
The corresponding evolution of the mean and variance in phase space is plotted in \cref{fig:Centering_Normal_MeanPhaseSpace}. 
In the uncontrolled case, all particles will (up to the effects of jumps and diffusion) remain on their initial orbit around the center.
In the control case, the orbit is immediately decreased in position, i.e. the particles move directly close to the center. On the other side in velocity, there is a small overshoot before converging to the center.
From \cref{fig:Centering_Norma_Mean}, we can obtain that the final configuration is kept stable.
In the uncontrolled case, one expects sinusoidal behaviour for the mean in position and velocity due to Hook's law given by $H$ (cf. \eqref{eq:coefficients}).
\begin{figure}[H]
	\begin{subfigure}[l]{0.3\textwidth}
		\includegraphics[width=\textwidth]{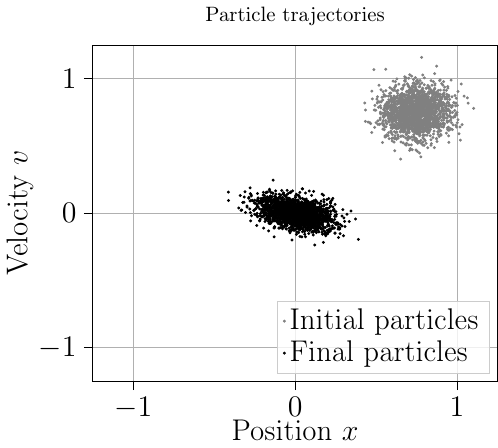}
		\caption{}
		\label{fig:Centering_Normal_InitialFinal}
	\end{subfigure}
	\hfill
	\begin{subfigure}[l]{0.3\textwidth}
		\includegraphics[width=\textwidth]{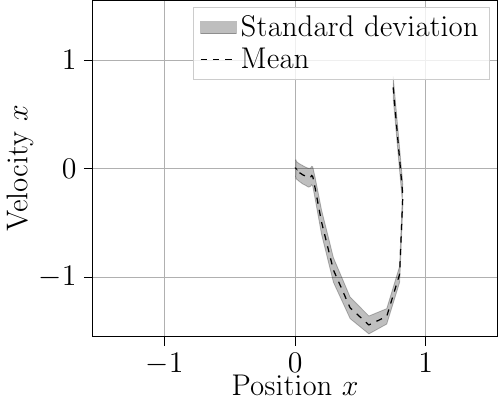}
		\caption{}
		\label{fig:Centering_Normal_MeanPhaseSpace}
	\end{subfigure}
	\hfill
	\begin{subfigure}[l]{0.3\textwidth}
		\includegraphics[width=\textwidth]{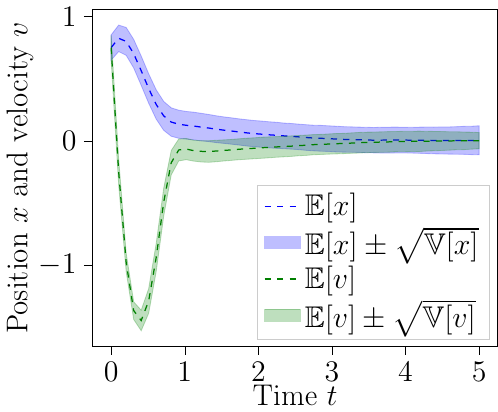}
		\caption{}
		\label{fig:Centering_Norma_Mean}
	\end{subfigure}
	\caption{Results of numerical experiments trying to center the particles starting from a normal distribution.
		(a) Initial (gray) and final (black) configuration;
		(b) Mean and standard deviation in phase space;
		(c) Mean and standard deviation in position and velocity over time.}
	\label{fig:Centering_NormalDistribution}
\end{figure}

Now, we start with a uniform distribution in a subset of our computation domain, i.e. uniform distributed in $[-1,1]^2$.
Our goal is again to center the particles in the middle, i.e. $\zd(t) = (0,0)$, while again considering $\eta=1$ in \eqref{eq:coefficients}.
In \cref{fig:Centering_Uniform}, we show the results of this test case.
In \cref{fig:Centering_Uniform_initialFinal} the initial and final configuration is visualized.
From \cref{fig:Centering_Uniform_MeanVar} it is evident that also here the final configuration is stable.
In \cref{fig:Centering_Uniform_control}, we plot the values of $\mu(t)$ for $t=0$ to give an impression of how the control might look like.
In \cref{sec:Stabilization}, we will use the results of this test case the generate a stabilization control for general random initial data.

\begin{figure}[H]
	\begin{subfigure}[l]{0.3\textwidth}
		\includegraphics[width=\textwidth]{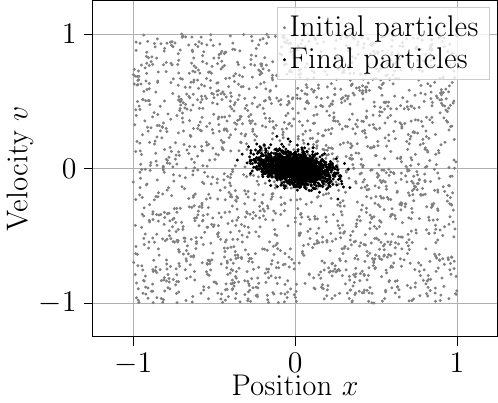}
		\caption{}
		\label{fig:Centering_Uniform_initialFinal}
	\end{subfigure}
	\hfill
	\begin{subfigure}[l]{0.3\textwidth}
		\includegraphics[width=\textwidth]{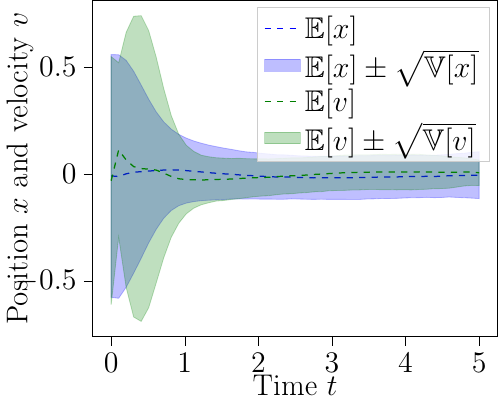}
		\caption{}
		\label{fig:Centering_Uniform_MeanVar}
	\end{subfigure}
	\hfill
	\begin{subfigure}[l]{0.32\textwidth}
		\includegraphics[width=\textwidth]{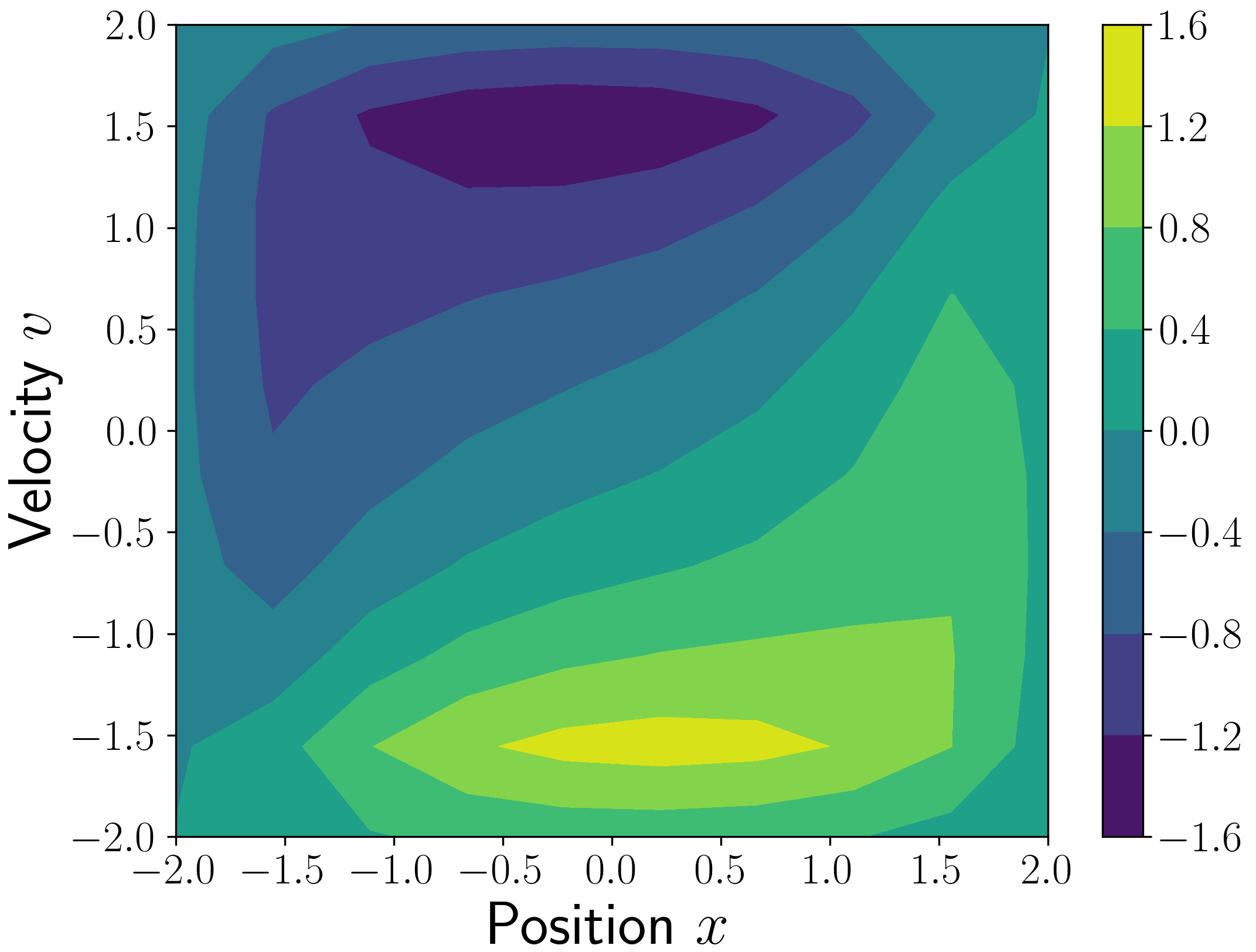}
		\caption{}
		\label{fig:Centering_Uniform_control}
	\end{subfigure}
	\caption{Results of numerical experiments trying to center the particles starting from a uniform distribution.
		(a) Initial and final configuration;
		(b) Evolution of mean and variance;
		(c) Control $\bmu(t)$ at final time $t=T$.}
	\label{fig:Centering_Uniform}
\end{figure}

\subsection{Stabilization}
\label{sec:Stabilization}

In this test case, we aim at stabilizing the system at $\zd(t)=(0,0)$ with  a random initial condition, in particular with an initial condition different from the one that was used to calculate the control.
More specifically, we show that by taking the average in time of the control obtained using our procedure above, we might be able to apply it to any input data and observe a stabilizing behavior.
Since our control depends by construction on position and velocity and hence on the state of the system, we can interpret it as a feedback-like control.
The idea of taking the time average of an open-loop control to obtain a closed-loop one can already be found in \cite{BartschBorzi2024MOCOKIFeedbackStabilization}.
However, in the current work, there are now distribution functions involved.
We consider still $\eta=1$ in \eqref{eq:coefficients}. 
This means, in particular, that the desired configuration $\zd$ is not reached in the uncontrolled case since the particles stay on their initial orbit (up to perturbations by jump and diffusion).

We define the time-averaged control given the optimized control from \cref{sec:Centering} where we started with a uniform initial distribution
\begin{align}
	\bar{u}(x,v) 
	= \sum_{\ell=1}^L \bar{\bmu}_\ell \bphi^x_\ell(x) \bphi^v_\ell(v) 
	\qquad\qquad
	\text{ with }
	\bar{\bmu}_\ell \coloneqq \frac{1}{T} \int_0^T \bmu_\ell(t) \rmd t
	\label{eq:time_avg_control}.
\end{align}
Notice that this control $\bar{u}$ is now independent of time.
Hence it can also be used in an infinite time interval.
We want to point out that there is no additional optimization procedure executed.

In \cref{fig:Stabilization} we show the results of this test case.
We start with a superposition of a bimodal and uniform configuration as initial condition that is shown in \cref{fig:Stabilization_Initial}.
After applying our control $\bar{u}$ that is visualized in \cref{fig:Stabilization_control}, we end up with a final configuration shown in \cref{fig:Stabilization_final}.
We observe that our average control is capable of collecting the particles in the center and keeping them there.
Notice that we initialized the particles outside of the domain where we generated uniform particles (cf. \cref{fig:Centering_Uniform}).
Nevertheless, since our control $u$ is defined in whole $\Rb^2$, it is possible the stabilize all particles in the center $(0,0)$.
\begin{figure}[H]
	\begin{subfigure}[l]{0.3\textwidth}
		\includegraphics[width=\textwidth]{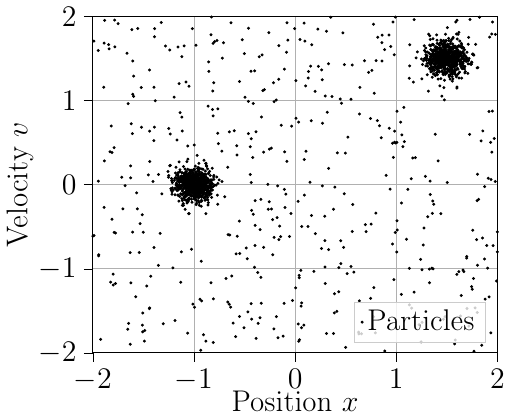}
		\caption{}
		\label{fig:Stabilization_Initial}
	\end{subfigure}
	\hfill
	\begin{subfigure}[l]{0.32\textwidth}
		\includegraphics[width=\textwidth]{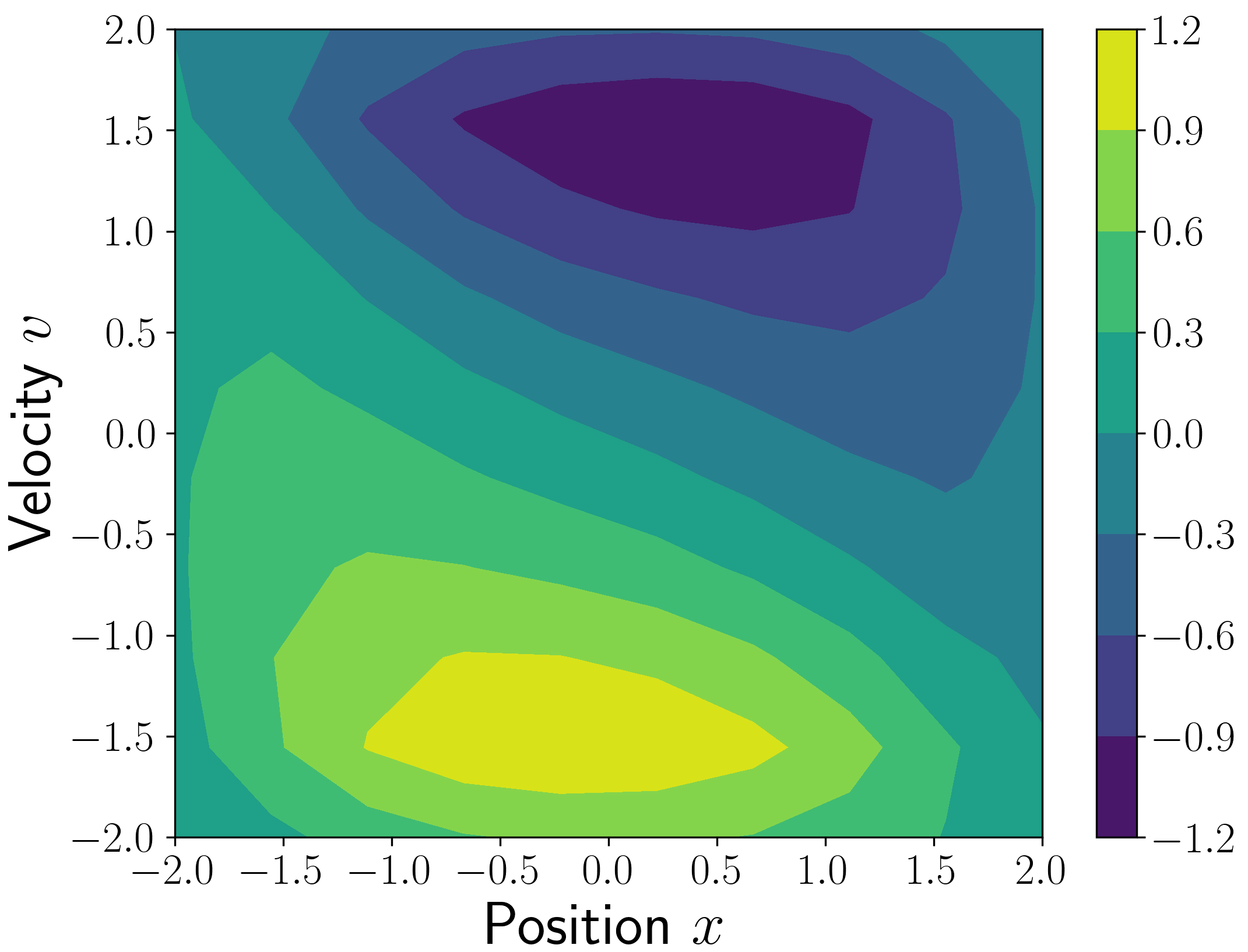}
		\caption{}
		\label{fig:Stabilization_control}
	\end{subfigure}
	\hfill
	\begin{subfigure}[l]{0.3\textwidth}
		\includegraphics[width=\textwidth]{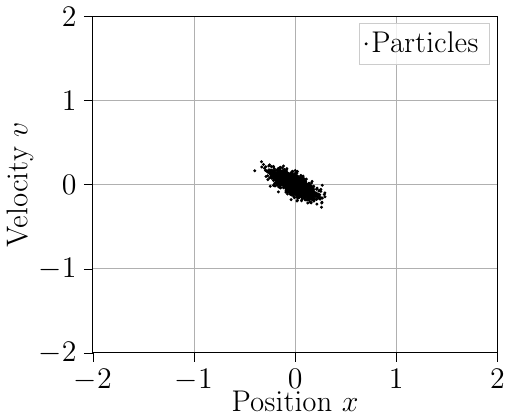}
		\caption{}
		\label{fig:Stabilization_final}
	\end{subfigure}
	\caption{Results of numerical experiments of stabilization with time-averaged control.
		(a) Initial configuration;
		(b) Contourplot of averaged $\bar{\bmu}$;
		(c) Final configuration obtained with averaged control $\bar{\bmu}$ without additional optimization process.}
	\label{fig:Stabilization}
\end{figure}

\subsection{Following a time-dependent trajectory}
\label{sec:TimeDependentTrajectory}

In the last test case, we search for an optimal control such that the particles follow a desired non-smooth time-dependent trajectory $\zd(t)$ for $t \in [0,T]$.
We choose 
$\zd(t) = (-\frac{\xmax}{2} + \frac{\xmax}{T}t, - |\frac{\vmax}{2T}t|+\frac{\vmax}{2})$.
Furthermore, we choose $\eta=0$ for this experiment and $\varepsilon_\varphi = 0.1$.
In particular, the particles will not feel a harmonic oscillator force, as this is not sensible for following a time-dependent trajectory that is not a circular motion.
We present the results of this test case in \cref{fig:Trajectory}, where in all plots the desired trajectory is plotted in red.
In \cref{fig:Initial_Final_Particles}, the initial configuration and final configurations are shown. In \cref{fig:MeanVar_PhaseSpace}, the desired trajectory and the resulting mean and variance of particles applying the optimal control are depicted in phase space, whereas in \cref{fig:MeanVar_PhaseSpace_Time} the mean and variance in position and velocity are plotted over time.
We see that the mean of the particles follows closely the prescribed desired trajectory. Since we cannot control the variance with out control mechanism, it grows as expected.

\begin{figure}[H]
	\begin{subfigure}[l]{0.3\textwidth}
		\includegraphics[width=\textwidth]{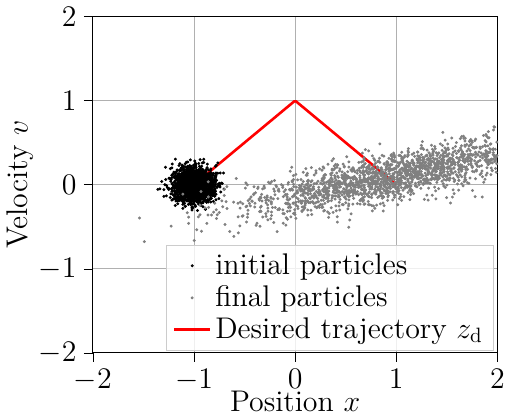}
		\caption{}
		\label{fig:Initial_Final_Particles}
	\end{subfigure}
	\hfill
	\begin{subfigure}[l]{0.3\textwidth}
		\includegraphics[width=\textwidth]{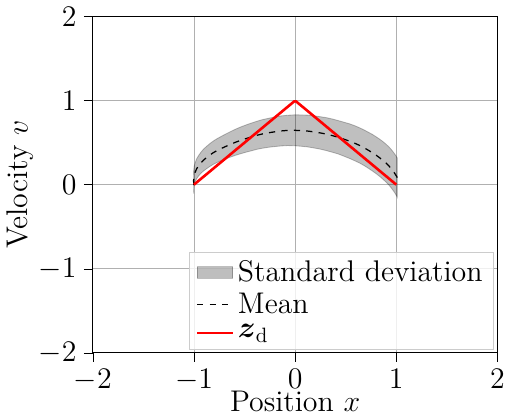}
		\caption{}
		\label{fig:MeanVar_PhaseSpace}
	\end{subfigure}
	\hfill
	\begin{subfigure}[l]{0.3\textwidth}
		\includegraphics[width=\textwidth]{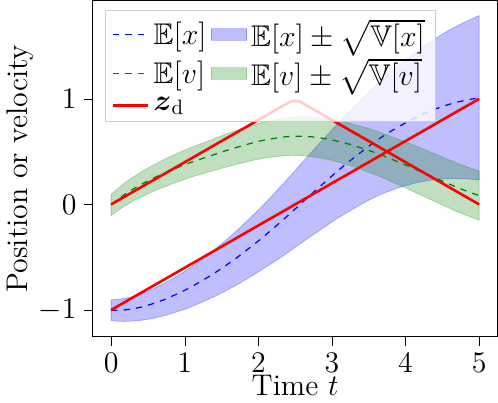}
		\caption{}
		\label{fig:MeanVar_PhaseSpace_Time}
	\end{subfigure}
	\caption{Results of numerical experiments with desired time-dependent trajectory (red).
		(a) Initial and final configuration;
		(b) Desired trajectory (red) and evolution of mean and standard deviation in phase space;
		(c) Desired trajectory (red) and evolution of mean and standard deviation over time.}
	\label{fig:Trajectory}
\end{figure}

\subsection{Interacting particles}
\label{sec:Coupling}

As a next example, we now consider a system of particles that interact with each other.
More in detail, we consider a system of $N$ particles, in which all nearest neighbours are coupled. For three particles with positions $x_i$, velocities $v_i$, $v_2$, $i=1,2,3$, and with positive parameters $\eta$ and $\omega$, we have the system:
\begin{align}
	\dot{x}_i = v_i
	\qquad\qquad
	\dot{v}_i = u(x_i,v_i,t) - \eta x_i - \omega\left(2x_i - \sum_{i=1,i\neq j}^3 x_i\right)
	\label{eq:interacting_model}
\end{align}

With this system of coupled oscillators, we can model in particular interacting phonons \cite{Koniakhin2023Phonons}.

As a test case, we initialize $N$ particles equidistantly distributed on an ellipse in phase space described by
\begin{align}
	\left( \frac{x}{A_x} \right)^2 + \left( \frac{v}{A_v} \right)^2 = 1,
	\label{eq:coupled_ellipse}
\end{align}
with given $A_x, A_v >0$.
The phonons should stay on this ellipse.
However, since we have diffusion and jumps in the process, the phonons leave this orbit.
With our control, we aim to keep the orbit as stable as possible.
For this reason, we implement also another functional $\Js_c$ for this test case with coupled particles:
\begin{align}
	\Js_{c}(z,t) = -\exp\left(-\frac{1}{2\sigma_\Js^2}\left(\left|\frac{x^2}{A_x^2}+\frac{v^2}{A_v^2}-1\right|^2\right)\right).
	\label{eq:phonon_functional_tracking}
\end{align}
This implements our desire to keep the particles on the ellipse described in \eqref{eq:coupled_ellipse}.

In the following figures, we present the results of this test case.
We set $A_x = 1.5$,  $A_v=1.7071067811865475$, and $\eta=1$, $\omega=0.5$.
\begin{figure}
	\begin{subfigure}[l]{0.49\textwidth}
		\centering
		\includegraphics[width=\textwidth]{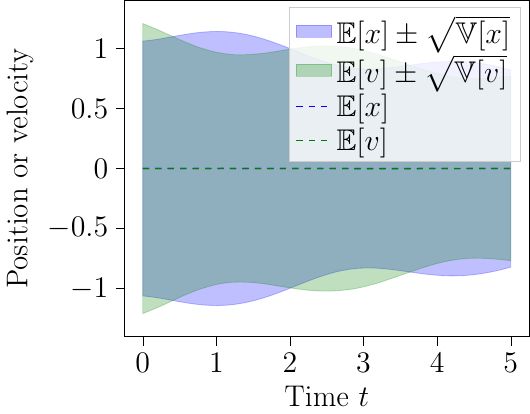}
		\caption{}
		\label{fig:uncontrolled_coupled}
	\end{subfigure}    
	\hfill
	\begin{subfigure}[l]{0.49\textwidth}
		\centering
		\includegraphics[width=\textwidth]{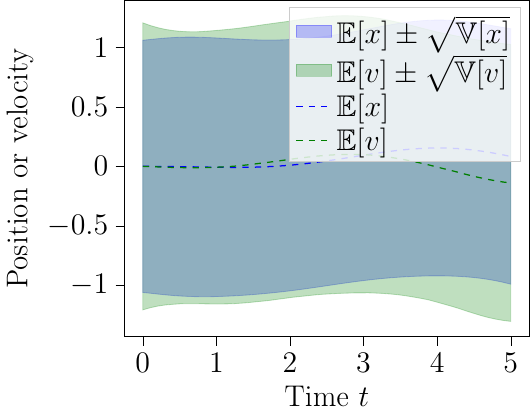}
		\caption{}
		\label{fig:controlled_coupled}
	\end{subfigure}    
	\caption{Results of experiment with coupled particles.
		(a): Uncontrolled case;
		(b): Applying optimized control mechanism.}
\end{figure}

\begin{figure}
	\begin{subfigure}[l]{0.49\textwidth}
		\centering
		\includegraphics[width=\textwidth]{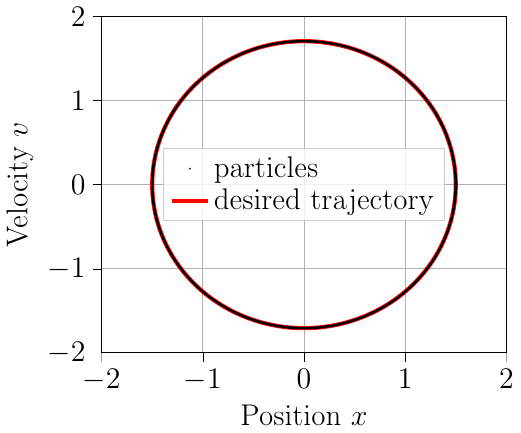}
		\caption{}
		\label{fig:initial_particles}
	\end{subfigure}
	\hfill
	\begin{subfigure}[l]{0.49\textwidth}
		\centering
		\includegraphics[width=\textwidth]{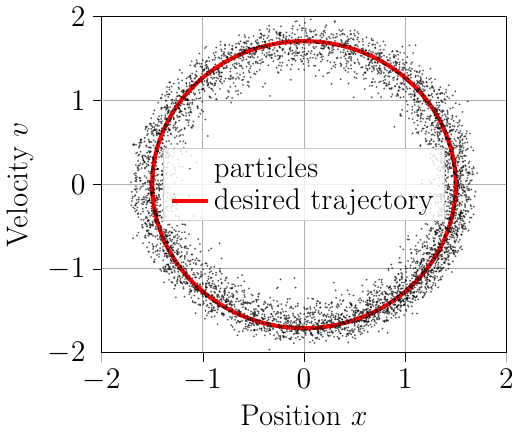}
		\caption{}
		\label{fig:terminal_particles}
	\end{subfigure}
	\caption{Behaviour of particles;
		(a): Initial particles on the desired ellipse;
		(b): final particles distributed on average on the ellipse.}
	\label{fig:particles_coupled}
\end{figure}

\begin{figure}
	\centering
	\begin{subfigure}[l]{0.49\textwidth}
		\includegraphics[width=\textwidth]{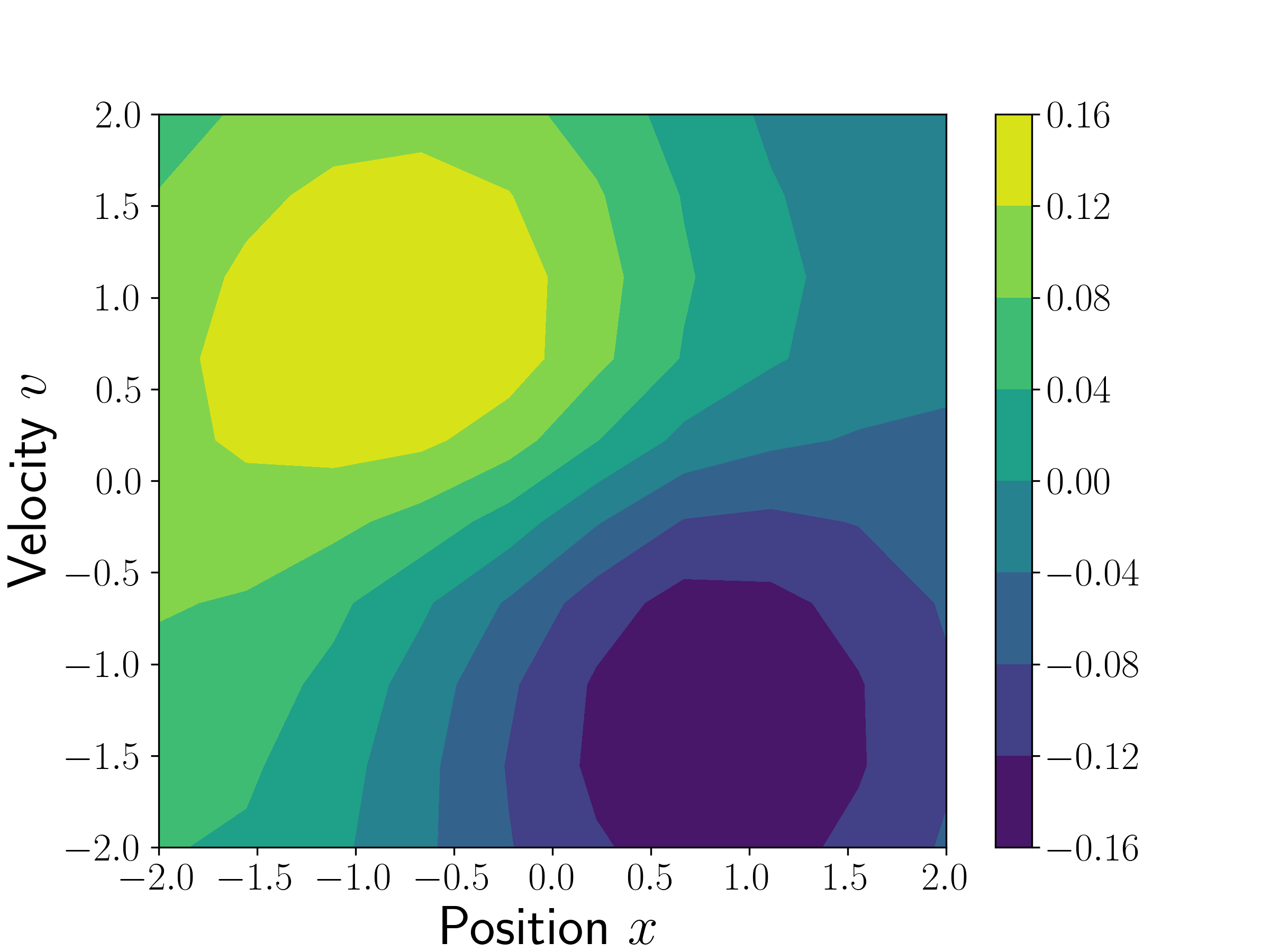}
		\caption{}
	\end{subfigure}
	\hfill
	\begin{subfigure}[r]{0.49\textwidth}
		\includegraphics[width=\textwidth]{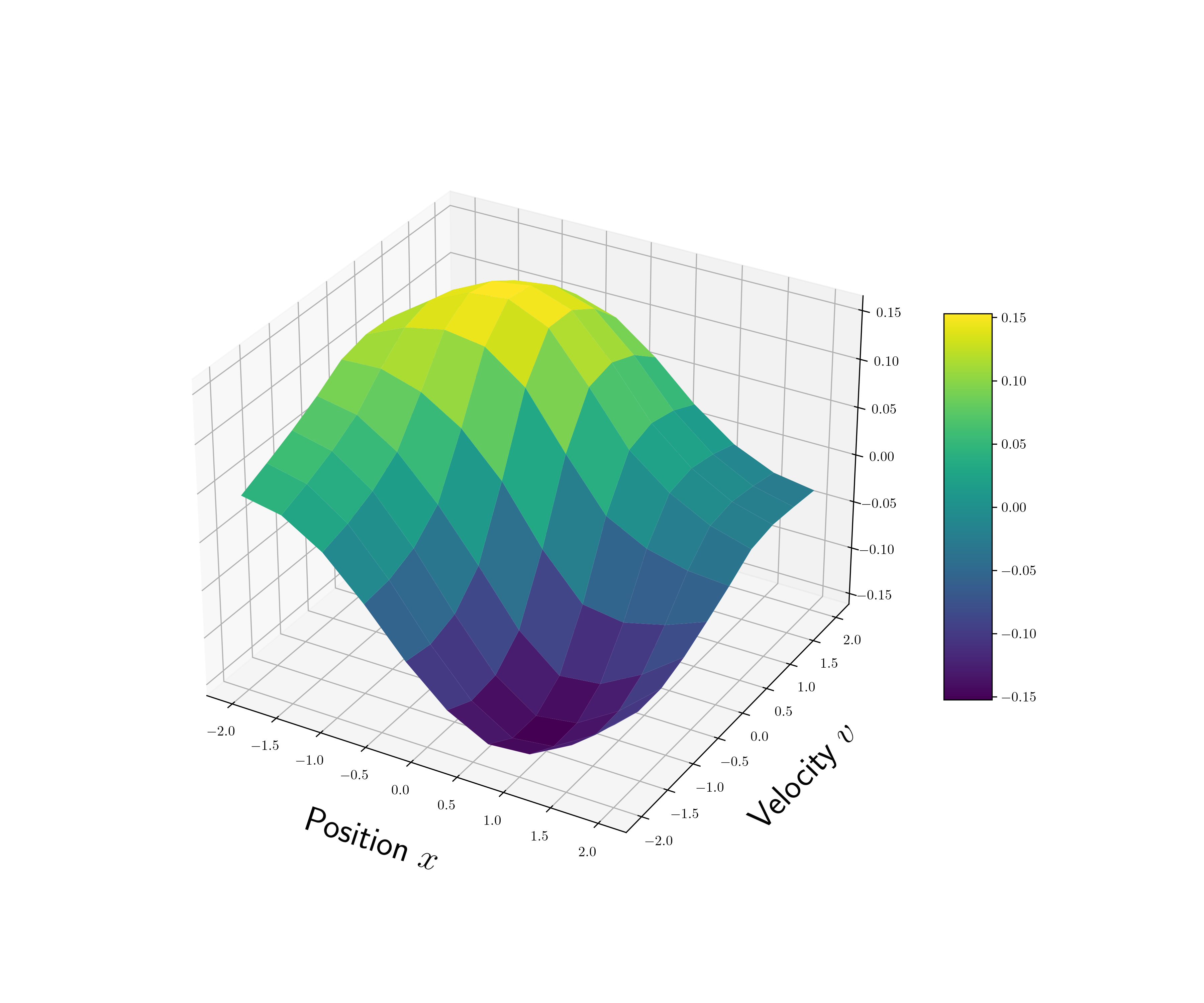}
		\caption{}
	\end{subfigure}
	\caption{Averaged (in time) optimal control;
	(a) Contour plot;
	(b) Surface plot}
	\label{fig:averaged_control}
\end{figure}

\newpage

\subsection{Adjoint equation independent of realization of forward trajectories}
\label{sec:AdjointIndependent}

In this section, we consider the same setting as in \cref{sec:TimeDependentTrajectory} but we do not take the solution of the forward model into account for the calculation of the solution to the adjoint model.
The main reason for this is that it is then possible to avoid the storing of the forward trajectories and hence save memory.
More specifically, it is possible to save the memory for $N \cdot N_t$ datapoints for the $N$ particles with $N_t$ timesteps.
Instead of using $\bz^h_j$, we use the (known) desired trajectory $\zd(t)$ and sample for each particle and realization a position and velocity according to $\mathcal{N}(\zd(t),\varsigma(t))$ with $\varsigma(t):[0,T] \rightarrow \Rb^+$ being the two-dimensional variance in phase space that mirrors the behavior of the variance of the forward model, i.e it growths linearly over time.

Notice, that in the generation of the gradient \eqref{eq:definition_Gradient}, we consider the forward trajectories for everything else expect the calculation of the adjoint variable $\br$ since this can be done during the calculation of the forward model.

In \cref{fig:Trajectory_independent}, we present the results of this test case. We observer a similar behavior as in the test case of \cref{sec:TimeDependentTrajectory} which stresses the effectiveness of our idea in this test case in order to save memory capacity.
When we compare the behavior of the convergence of the (relative) functional, we see that more steps are needed in order to reach a certain functional value compared the the case in which we take the exact forward trajectories (cf. \cref{fig:Comparison_Functional}).

\begin{figure}[H]
	\begin{subfigure}[l]{0.3\textwidth}
		\includegraphics[width=\textwidth]{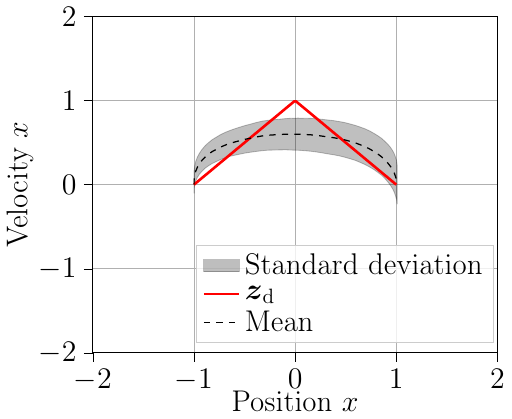}
		\caption{}
		\label{fig:MeanVar_PhaseSpace_independent}
	\end{subfigure}
	\hfill
	\begin{subfigure}[l]{0.3\textwidth}
		\includegraphics[width=\textwidth]{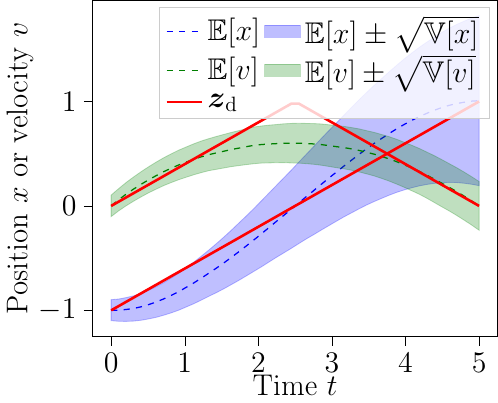}
		\caption{}
		\label{fig:MeanVar_PhaseSpace_TIME_independent}
	\end{subfigure}
	\hfill
	\begin{subfigure}[l]{0.3\textwidth}
		\includegraphics[width=\textwidth]{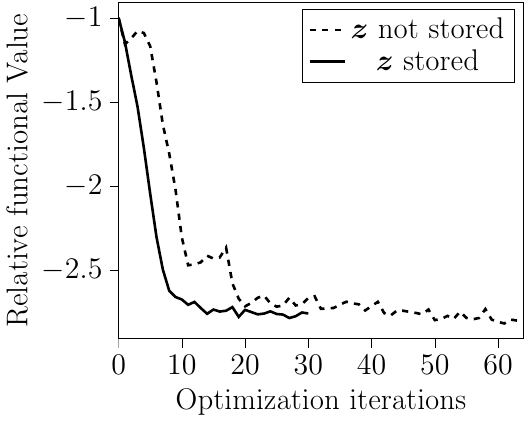}
		\caption{}
		\label{fig:Comparison_Functional}
	\end{subfigure}
	\caption{Results of numerical experiments with desired time-dependent trajectory (red).
		(a) Evolution of mean and standard deviation in phase space;
		(b) Evolution of mean and standard deviation for position and velocity in time;
		(c) Comparison of convergence of relative functional $\nicefrac{\hat{j}^h(u^\ell)}{\hat{j}^h(u^0)}$ over optimization iterations.}
	\label{fig:Trajectory_independent}
\end{figure}

\section{Conclusion}
\label{sec:Conclusion}

In this work, we developed and analyzed an adjoint-based optimization method for the control of systems governed by jump-diffusion processes, staying within the microscopic framework of stochastic differential equations. Our approach avoids the curse of dimensionality associated with macroscopic PDE-based methods and allows for highly parallelizable computations.

We derived the optimality system using a discretize-then-optimize approach and validated the theoretical results with Monte Carlo methods tailored to the microscopic context. Numerical experiments demonstrated the effectiveness of the proposed method in diverse scenarios, including centering particles, stabilization, and following time-dependent trajectories and using interacting and non-interacting particles. 
Moreover, we also presented a memory saving strategy that avoids storing the forward trajectories.
These results underscore the robustness and versatility of our approach in solving complex control problems in stochastic settings.

Future work could explore extensions to more complex jump processes including the control also in the diffusive and jump part. The integration of machine learning for enhanced control strategies also presents an interesting direction for further research.

\section*{Acknowledgments}
\noindent
This work was partially funded by the Deutsche Forschungsgemeinschaft within SFB 1432, Project-ID 425217212.
The authors acknowledge the Young Starting Fund of the University of Konstanz for partially funding J.R.

\bigskip

\paragraph*{\textbf{Declaration of generative AI and AI-assisted technologies in the writing process.}}

During the preparation of this work the authors used ChatGPT in order to improve the style of the writing. After using this tool, the authors reviewed and edited the content as needed and take full responsibility for the content of the published article.\\

\paragraph*{\textbf{Author contribution.}}
J.B. procured funding and led the project.
A.B. initiated the project and gave crucial ideas to improve the work.
G.C. initiated the project and worked together with J.B. to derive the theoretical results.
J.R. implemented the code and conducted numerical experiments under the supervision of J.B.
All authors edited and revised the manuscript.\\

\paragraph*{\textbf{Conflict of interest.}} The authors have no relevant financial or non-financial interests to disclose.

\bibliographystyle{acm}

\end{document}